\documentclass[12pt,reqno]{amsart}
\usepackage[a4paper,left=1in,right=1in,top=1in,bottom=1in,footskip=8mm]{geometry}
\usepackage{mathrsfs,amsfonts,amssymb}
\usepackage{amsmath}
\usepackage[utf8]{inputenc}
\usepackage{graphicx,indentfirst}
\usepackage{stmaryrd,cite}
\usepackage{cases}
\usepackage{enumerate}
\newcommand{\norm}[1]{\Vert#1\Vert}

\newcommand{\R}{\mathbb R}
\newcommand{\C}{\mathbb C}
\newcommand{\proj}{\mathbb{P}}
\newcommand{\im}{\mathrm{i}}
\newcommand{\RE}{\mathrm{Re}\,}
\newcommand{\D}{\partial}
\newcommand{\eps}{\varepsilon}

\newcommand{\A}{\mathcal{A}}
\newcommand{\dv}{\mathrm{div}\,}

\newcommand{\nb}{\nabla}
\newcommand{\vv}{\upsilon}

\newcommand{\tr}{\mathrm{tr}}

\newcommand{\ls}{\leqslant}
\newcommand{\gs}{\geqslant}
\newcommand{\N}{\mathcal{N}}
\newcommand{\NN}{\mathbb{N}}

\newcommand{\no}{\nonumber}

\newcommand{\h}{\mathcal{H}}

\newcommand{\sur}{{\Gamma_t}}

\newcommand{\etaD}[1][\D]{#1_\eta} 
\newcommand{\lsi}[2][\D]{#1_{#2}} 
\usepackage{listings}
\usepackage{ifpdf}
\ifpdf
\usepackage[
         hyperindex=true,
         pdfstartview=FitH,
         bookmarksnumbered=true,
         bookmarksopen=true,
         pdfborder=001,
         pdfauthor={Chengchun Hao},
         pdftitle={},
         pdfkeywords={},
         ]{hyperref}
\else
\usepackage[
         hypertex,
         hyperindex,
         linkcolor=blue,
         unicode,
         citecolor=blue%
         ]{hyperref}
\fi

\allowdisplaybreaks
\numberwithin{equation}{section}
\usepackage{amsthm}
\newtheorem{theorem}{Theorem}[section]
\newtheorem{corollary}[theorem]{Corollary}
\newtheorem{lemma}[theorem]{Lemma}
\newtheorem{proposition}[theorem]{Proposition}
\theoremstyle{definition}

\theoremstyle{remark}

\begin{document}
\title[Ill-posedness of ideal MHD]{Ill-posedness of free boundary problem of the incompressible ideal MHD}

\author{Chengchun Hao}

\address{Institute of Mathematics,
Academy of Mathematics and Systems Science,
and Hua Loo-Keng Key Laboratory of Mathematics,
Chinese Academy of Sciences,
Beijing 100190, China}
\email{hcc@amss.ac.cn}

\author{Tao Luo}

\address{Department of Mathematics,
	City University of Hong Kong,
	Hong Kong, China}
\email{taoluo@cityu.edu.hk}

\maketitle
\begin{abstract}
	In the present paper, we  show the ill-posedness of the free boundary problem of the incompressible ideal magnetohydrodynamics (MHD)
	equations  in two spatial dimensions for any positive vacuum permeability $\mu_0$, in Sobolev spaces. The analysis is uniform for any  $\mu_0>0$. 
\end{abstract}

\section{Introduction}
This paper is concerned with the ill-posedness of the following free boundary problem for the incompressible ideal MHD equations
\begin{subequations}\label{e'}
	\begin{numcases}{}
	\D_t \vv +\vv \cdot \nabla \vv  +\nb q =\frac{1}{\mu_0}\tilde{H} \cdot\nabla \tilde{H},  &in $\Omega_t$, \label{e'.1}\\
	\D_t \tilde{H} +\vv\cdot \nabla \tilde{H}=\tilde{H}\cdot\nabla \vv,  &in $\Omega_t$, \label{e'.2}\\
	\dv \vv=0, \quad \dv \tilde{H}=0,  &in $\Omega_t$, \label{e'.3}\\
	\D_t\Gamma(t)=\vv\cdot \N,\quad \tilde{H}\cdot \N=0,\quad q=0,   &on $\sur$,
	\label{e'.4}\\
	\vv(0,x)=\vv_0(x),\quad \tilde{H}(0,x)=\tilde{H}_0(x), &$x\in \Omega:=\Omega_0$,
	\end{numcases}
\end{subequations}
where $\vv$ is the velocity field, $\tilde{H}$ is the magnetic field, $q$ is the total pressure and $\mu_0>0$ is the vacuum permeability; $\N$ is the unit outward  normal vector on the boundary $\Gamma_t$, $\Omega_0$ is the initial bounded domain. 

Without the magnetic field, i.e. $\tilde{H}=0$ in \eqref{e'}, the problem \eqref{e'} reduces to the free boundary problem of incompressible 
Euler equations, for which  the local-in-time well-posedness in Soblev spaces was obtained first in \cite{wu1, wu2} for the irrotational case. Extensions including the case without irrotationalilty assumption have been made in   \cite {ABZ},  \cite{AM}, \cite{CL00}, \cite{Coutand}, \cite{HLarma}, \cite{HaoWang}, \cite{L1}, 
\cite{Lindblad2}, \cite{LN}, \cite{luozeng}, \cite{SZ} and \cite{zhang}. The Taylor sign condition, $\nabla_{\N} p<0$ on 
$\partial \Omega_t$,  plays a crucial role  in establishing the above well-posedness results for Euler equations, where $p$ is the fluid pressure for Euler equations. It was found by Ebin \cite{Ebin} that the free boundary problem for incompressible Euler equations is ill-posed when the Taylor sign condition fails. A natural question for MHD free boundary problem is that if this is still the case or the magnetic field has some stablizing effect, as discussed for the current-vortex sheet problem of MHD in \cite{chenwang, SWZ, WangYu13}.   A stability condition, 
\begin{equation}\label{ts}
\nabla_{\N} q<0, \ \text{ on }  \partial \Omega_t
\end{equation}  
was identified in \cite{HLarma} for the MHD free boundary problem \eqref{e'} under which the a priori estimates are derived.  The present paper is concerned with  what will happen if this condition is violated. 

In 2-spatial dimension case, there is a particular steady solution to \eqref{e'} with pure rotation:  For $t\gs 0$, $\Omega_t=\{x\in \R^2:\ |x|\ls 1\}=B_1(0)$,  the unit disk, $\vv(t, x)=(-x_2,\  x_1)=x^{\perp}$, $\tilde H(t, x)=b\vv (t,\ x)$ (for $x=(x_1,\  x_2)$ and $b\in\R$). In this case, $q$ is solved by
the following Dirichlet problem:
$$\Delta q=2\left(1-\frac{b^2}{\mu_0}\right),\ {\rm in}\ B_1(0); \  q=0, \ {\rm on} \ \partial B_1(0).$$
For this solution, it is easy to verify that, on $\partial B_1(0)$,
\begin{equation}\label{taylor}
\nabla_{\N} q=\left(\frac{b^2}{\mu_0}-1\right)(\vv\cdot \nabla \vv)\cdot x=1-\frac{b^2}{\mu_0}, \end{equation}
by noting that $\N=x$ on $\partial B_1(0)$. Therefore, the Taylor sign condition \eqref{ts} fails in this case when $\mu_0\gs b^2$.  A natural question arises: For $\mu_0\gs b^2$, is the free boundary problem 
\eqref{e'} still well-posed? We prove in this paper that this is not 
the case when $\mu_0 \gg b^2$ in the following sense: 

\textit{we can construct a family of initial data for the problem \eqref{e'} which 
	converges to the above particular steady rotation solution. However, as long as $t>0$, the family of solutions to \eqref{e'} with those initial data diverges in some Sobolev $H^{\mu}$-norm for $\mu\gs 2$.} 

It should be also noted that our analysis is uniform in $\mu_0>0$. 

Our construction is strongly motivated by that of Ebin \cite{Ebin}
for incompressible Euler equations. This approch  involves showing that the linearized problem about the particular steady rotation solution has solutions which exhibit rapid exponential growth and the actual nonlinear problem behaves like the linearized one. 

There are few results for the posedness of the free boundary problem of incompressible ideal MHD equations.  A priori estimates for this problem were derived in \cite{HLarma} with a bounded initial domain homeomorphic to a ball, provided that the size of the magnetic field to be invariant on the free boundary. A local existence result was proved in \cite{GW} for the initial flat domain.     The plasma-vacuum system was investigated  in \cite{Hao17} where the a priori estimates were derived. For the special case where the magnetic field is zero on the free boundary and in vacuum,  the local existence and uniqueness of the free boundary problem of incompressible viscous-diffusive MHD flow in three-dimensional space with infinite depth setting was proved in \cite{Lee} where also a local unique solution was obtained  for the  free boundary MHD without kinetic viscosity and magnetic diffusivity via zero kinetic viscosity-magnetic diffusivity limit. 
For the incompressible viscous MHD equations, a free boundary problem  in a simply connected domain of $\R^3$ was studied by a linearization technique and the construction of a sequence of successive approximations in \cite{PS10} with an irrotational condition for magnetic fields in a part of the domain. 
The well-posedness of the linearized plasma-vacuum interface problem in incompressible ideal MHD was studied in \cite{MTT14} in an unbounded plasma domain. For other related results of  MHD equations with free boundaries or interfaces, one may refer to \cite{chenwang, Lee2,  ST, Trakhinin, WangYu13, WangXin17}.

\section{Lagrangian Description}
For simplicity, we denote $H=\mu_0^{-1/2}\tilde{H}$. Then, the problem \eqref{e'} reduces to the case $\mu_0=1$, i.e., 
\begin{subequations}\label{e}
	\begin{numcases}{}
	\D_t \vv +\vv \cdot \nabla \vv  +\nb q =H \cdot\nabla H,  &in $\Omega_t$, \label{e.1}\\
	\D_t H +\vv\cdot \nabla H=H\cdot\nabla \vv,  &in $\Omega_t$, \label{e.2}\\
	\dv \vv=0, \quad \dv H=0,  &in $\Omega_t$, \label{e.3}\\
	\D_t\Gamma(t)=\vv\cdot \N,\quad H\cdot \N=0,\quad q=0,   &on $\sur$,
	\label{e.4}\\
	\vv(0,x)=\vv_0(x),\quad H(0,x)=H_0(x), &$x\in \Omega:=\Omega_0$.
	\end{numcases}
\end{subequations}

We transform the system \eqref{e} into Lagrangian variables (e.g. \cite{Morr09}). We denote $\eta=\eta(t,a)=(\eta^1,\eta^2)$  the position of a fluid element or parcel at time $t$ with $a=(a^1,a^2)$ being the fluid element label,  defined to be the position of the fluid element at the initial time, $a=\eta(0,a)$, but this is  not  necessarily always  the case. Denote $\Omega_t$  the domain occupied by the fluid at time $t$, then $\eta:\Omega_0\to\Omega_t$ is assumed to be 1-1 and onto, at each fixed time $t$.

Let $\D \eta^i/\D a^j=: \eta^i_{,j}$  be  the deformation matrix ,  its Jacobian determinant $J:=\det(\eta^i_{,j})$ is given by 
\begin{align*}
J=\frac{1}{2}\eps_{kj}\eps^{il}\eta^k_{,i}\eta^j_{,l},
\end{align*}
where $\eps_{ij}=\eps^{ij}$ is the two-dimensional unit, purely antisymmetric, Levi-Civita tensor density.  In this notation, 
\begin{align}\label{Jac}
d\eta=Jda,
\end{align}
and components of an area form map according to
\begin{align}\label{JacS}
(dS(q))_i=Ja^j_{,i}(dS(a))_j,
\end{align}
where $Ja^j_{,i}$ is the transpose of the cofactor matrix of $\eta^j_{,i}$,  given by
\begin{align*}
Ja^j_{,i}=\eps_{ik}\eps^{jl}\eta^{k}_{,l}.
\end{align*}

Clearly,  $\dot{\eta}(t,a)=\vv(t,x)$, where $\dot{\eta}(t,a)=\frac{\partial \eta(t, a)}{\partial t}$. The label of the element will given by $a=\eta^{-1}(t,x)=:a(t,x)$. Thus, the Eulerian velocity field is given by 
\begin{align*}
\vv(t,x)=\dot{\eta}(t,a)|_{a=a(t,x)}.
\end{align*}
For an incompressible fluid,  $J=1$. Thus, we can invert $\eta=\eta(t,a)$ to obtain $a=a(t,\eta)$, 
\begin{align}\label{etainverse}
\eta^i_{,k}a^k_{,j}=a^i_{,k}\eta^k_{,j}=\delta^i_j,
\end{align}
where $a^k_{,j}=\D a^k/\D \eta^j$  (repeated indices are summed), and the components of the Eulerian gradient are given by
\begin{align*}
\left.\frac{\D}{\D x^k}=a^i_{,k}\frac{\D}{\D a^i}\right|_{a=a(t,x)}.
\end{align*}
Thus, for $x=\eta(t,a)$ and any function $f(t,a)=\tilde{f}(t,x)=\tilde{f}(t,\eta(t,a))$, 
\begin{align*}
\dot{f}|_{a=a(t,x)}=\frac{\D \tilde{f}}{\D t}+\dot{\eta}(t,a)\left.\frac{\D \tilde{f}}{\D x^i}\right|_{a=a(t,x)}=\frac{\D \tilde{f}}{\D t}+\vv\cdot\nb \tilde{f}(t,x).
\end{align*}

For ideal MHD, a magnetic field, $H_0(a)$, can be attached to a fluid element, and then the frozen flux condition yields $H\cdot dS(x)=H_0\cdot dS(a)$, and from \eqref{JacS} we obtain
\begin{align}\label{H.sol}
H^i=\eta^i_{,j}H_0^j,
\end{align}
which can be also obtained from the equation \eqref{e.2}. In fact, we have by \eqref{e.2} 
\begin{align*}
\dot{H}^j=H^k a^i_{,k}\vv^j_{,i}=H^k a^i_{,k}\dot{\eta}^j_{,i}=-H^k \dot{a}^i_{,k}\eta^j_{,i},
\end{align*}
due to
\begin{align*}
\dot{\eta}^i_{,k}a^k_{,j}=-\eta^i_{,k}\dot{a}^k_{,j},
\end{align*}
from \eqref{etainverse}, and  by \eqref{etainverse} again
\begin{align*}
\dot{H}^ja^l_{,j}=-H^k \dot{a}^i_{,k}\eta^j_{,i}a^l_{,j}=-H^k \dot{a}^l_{,k},
\end{align*}
and then
\begin{align*}
\frac{d}{dt}(H^ja^l_{,j})=0,
\end{align*}
which yields the desired identity \eqref{H.sol}. 

It follows from \eqref{e.1}, \eqref{e.3} and $q=0$ on $\Gamma_t=\eta(\Gamma)$,  that
\begin{subequations}\label{q}
	\begin{numcases}{}
	\Delta q=\tr(D H)^2-\tr(D\vv)^2, &in $\Omega_t$,\\
	q=0, &on $\Gamma_t$,
	\end{numcases}
\end{subequations}
where $\tr(D\vv)^2:=\D_i\vv^j\D_j\vv^i$ and $\tr(D H)^2:=\D_iH^j\D_jH^i$ are the trace of the square of the matrices of differentiations  of $\vv$ and $H$, respectively.  This elliptic Dirichlet boundary value problem admits a unique solution, so $\Delta^{-1}$ is well-defined. Here $\Delta^{-1}g=f$ in a domain means $\Delta f=g$ in this domain and $f=0$ on the boundary of this domain. Hence,
\begin{align}\label{q.sol}
q=\Delta^{-1}(\tr(D H)^2-\tr(D\vv)^2).
\end{align}

Therefore,  \eqref{e.1} can be rewritten as 
\begin{align}\label{eq.eta}
\ddot{\eta}=(\nb (\Delta^{-1}(\tr(D\vv)^2-\tr(D H)^2)))\circ \eta+(H\cdot\nb H)\circ\eta.
\end{align}

Note that $\vv(t,x)=\dot{\eta}(t,a(t,x))$ and \eqref{H.sol}, so \eqref{eq.eta} is of the form
\begin{align}\label{eqn}
\ddot{\eta}=Z(\eta,\dot{\eta}).
\end{align}
Then we study the initial value problem for \eqref{eqn} with the initial data $\eta(0)=\eta_0$ and $\dot{\eta}(0)$.
\section{An example} \label{sec.examp}

As in \cite{Ebin}, we consider a disc of MHD fluid spinning at constant angular velocity.  Let $\Omega$ be the unit disc in $\R^2$. For convenience, we identify  points in $\R^2$  with the complex numbers $\C$.  Then $\eta(0)=\eta_0$ is the inclusion of $\Omega$ in $\C$. We choose  $\dot{\eta}(0)$ to be a $\pi/2$ rotation. Thus, for $z\in \Omega$ being a complex variable, $\eta(0,z)=z$ and $\dot{\eta}(0,z)=\im z$. Since $H_0\cdot\N=0$ on the boundary, we can take $H_0(\eta(z))=\im bz$ for some constant $b\in \R$ and $|b|<1$. 

As we will demonstrate, the solution to \eqref{eq.eta} with the above initial data is $\eta(t,z)=e^{\im t}z$ and $H(t,z)=\im bz$.  We have $\dot{\eta}(t,z)=\im \eta(t,z)=\im e^{\im t}z$, and $a(t,z)=e^{-\im t}z$. Thus, $\vv(t,z)=\dot{\eta}(t,a(t,z))=\im e^{\im t}a(t,z)=\im e^{\im t}e^{-\im t}z=\im z$. By using real variables, we find that $D\eta=e^{\im t}I$ where $I$ is the $2\times 2$  unit matrix $(\delta^i_j)$. Thus, by \eqref{H.sol}, we have $H(t,z)=\im b z$. Moreover,  $\vv(x_1,x_2)=(-x_2,x_1)$ and $H(x_1,x_2)=b(-x_2,x_1)$, both  satisfy the divergence-free condition, and we get the matrices
$$D\vv=\left(\begin{matrix}
0&-1\\1&0
\end{matrix}\right),\quad \text{and} \;
DH=\left(\begin{matrix}
0&-b\\b&0
\end{matrix}\right).$$
Thus, $\tr(D\vv)^2=-2$ and $\tr(DH)^2=-2b^2$, and the Dirichlet problem is solved by $q=(b^2-1)\Delta^{-1}(-2)=\frac{1-x_1^2-x_2^2}{2}(b^2-1)$. Then, $\nb q=(1-b^2)(x_1,x_2)=(1-b^2)z$. It is easy to check that $H\cdot \nb H=-b^2z$ by using real variables.

Since $\dot{\eta}(t,z)=\im \eta(t,z)$ and $\ddot{\eta}(t,z)=-\eta(t,z)$, we get
\begin{align*}
\ddot{\eta}(t,z)+\nb q(t,\eta(t,z))=(H\cdot\nb H)(\eta(t,z)),
\end{align*}
namely, $\eta(t,z)$ satisfies \eqref{eq.eta}.  In this example we note that, both the velocity $\vv(t,z)=\im z$ and the magnetic field $H(t,z)=\im bz$ are independent of the time $t$.

\section{Reformulation in Lagrangian variables}

We first recall some operators defined in \cite{Ebin}. The operator $\Re_\eta:C^\infty(\eta(\Omega))\to C^\infty(\Omega)$ defined by $\Re_\eta(f)=f\circ \eta=f(\eta)$. It is obvious that the inverse of $\Re_\eta$ is $\Re_{\eta^{-1}}$ given by $\Re_{\eta^{-1}}(g)=g\circ {\eta^{-1}}=g({\eta^{-1}})$ where $\eta^{-1}(t,x)=a(t,x)$.

For a differential operator $P$, we set $\etaD[P]=:\Re_\eta P \Re_{\eta^{-1}}$.  For examples,  the operator $\etaD[D]=\Re_\eta D\Re_{\eta^{-1}}: C^\infty(\Omega)\to C^\infty(\Omega)$, where $D$ is the total derivative, 
and the operator $\etaD[\nb]=\Re_\eta \nb\Re_{\eta^{-1}}$ where $\nb$ is the gradient. Define 
\begin{align*}
K(\eta)=\Re_\eta \Delta \Re_{\eta^{-1}}: C_0^\infty(\Omega)\to C^\infty(\Omega)
\end{align*}
where $C_0^\infty(\Omega)=\{f\in C_0^\infty(\Omega): f|_{\D\Omega}=0 \}$. Namely, $K(\eta)=\etaD[\Delta]$. $K(\eta)$ is invertible because so does $\Delta$. In fact,
\begin{align*}
K(\eta)^{-1}=\Re_\eta \Delta^{-1} \Re_{\eta^{-1}}.
\end{align*}

Since $\vv=\dot{\eta}\circ\eta^{-1}$, we have 
\begin{align*}
D\vv\circ\eta=\Re_\eta D\vv=\Re_\eta D(\dot{\eta}\circ \eta^{-1})=\Re_\eta D\Re_{\eta^{-1}}\dot{\eta}=\etaD[D] \dot{\eta}.
\end{align*}
Similarly, from \eqref{H.sol}, it follows 
\begin{align*}
\D_iH^j(t,x)=&\D_i(\eta^j_{,k}(t,a(t,x))H^k_0(x)\\
=&\eta^j_{,kl}(t,a(t,x))a^l_{,i}(t,x)H^k_0(x)+\eta^j_{,k}(t,a(t,x))H^k_{0,i}(x),\\
\tr(DH)^2=&(\eta^j_{,kl}(t,a(t,x))a^l_{,i}(t,x)H^k_0(x)+\eta^j_{,k}(t,a(t,x))H^k_{0,i}(x))\\
&\cdot(\eta^i_{,mn}(t,a(t,x))a^n_{,j}(t,x)H^m_0(x)+\eta^i_{,m}(t,a(t,x))H^m_{0,j}(x))\\
=&\eta^j_{,kl}(t,a(t,x))a^l_{,i}(t,x)H^k_0(x)\eta^i_{,mn}(t,a(t,x))a^n_{,j}(t,x)H^m_0(x)\\
&+2\eta^j_{,kl}(t,a(t,x))H^k_0(x)H^l_{0,j}(x)\\
&+\eta^j_{,k}(t,a(t,x))H^k_{0,i}(x)\eta^i_{,m}(t,a(t,x))H^m_{0,j}(x)\\
=&\tr((\eta^j_{,kl}\circ\eta^{-1}) (\eta^{-1})^l_{,i}H^k_0)^2+\tr((\eta^j_{,kl}\circ\eta^{-1})(H^k_0H^l_0)_{,i})\\
&+\tr((\eta^j_{,k}\circ\eta^{-1})H^k_{0,i})^2,
\end{align*}
and
\begin{align*}
(H\cdot\nb H)_k=&\delta_{ki}H^ja^l_{,j}H^i_{,l}=\delta_{ki}\eta^j_{,m} H^m_0a^l_{,j}(\eta^i_{,n}H^n_0)_{,l}=\delta_{ki} H^l_0(\eta^i_{,n}H^n_0)_{,l}\\
=&\delta_{ki} H^l_0(\eta^i_{,nm}a^m_{,l}H^n_0+\eta^i_{,n}H^n_{0,l})\\
=&\delta_{ki}(\eta^i_{,nm}\circ\eta^{-1})(\eta^{-1})^m_{,l} H^l_0H^n_0+\delta_{ki} (\eta^i_{,n}\circ\eta^{-1})H^n_{0,l}H^l_0.
\end{align*}
Then, from \eqref{eq.eta}, we get
\begin{align*}
Z(\eta,\dot{\eta})=&(\nb (\Delta^{-1}\Re_{\eta^{-1}}(\tr(\etaD[D] \dot{\eta})^2)\circ \eta\\
&-(\nb \Delta^{-1}\tr((\eta^j_{,kl}\circ\eta^{-1}) (\eta^{-1})^l_{,i}H^k_0)^2)\circ \eta\\
&-(\nb \Delta^{-1}\tr((\eta^j_{,kl}\circ\eta^{-1})(H^k_0H^l_0)_{,i}))\circ \eta\\
&-(\nb \Delta^{-1}\tr((\eta^j_{,k}\circ\eta^{-1})H^k_{0,i})^2)\circ \eta\\
&+ ((\eta_{,nm}\circ\eta^{-1})(\eta^{-1})^m_{,l} H^l_0H^n_0)\circ \eta\\
&+ ((\eta_{,n}\circ\eta^{-1})H^n_{0,l}H^l_0)\circ \eta\\
=&\etaD[\nb]K(\eta)^{-1}\tr(\etaD[D] \dot{\eta})^2-\etaD[\nb]K(\eta)^{-1}\tr(\eta^j_{,kl} \Re_\eta((\eta^{-1})^l_{,i}H^k_0))^2\\
&-\etaD[\nb]K(\eta)^{-1}\tr(\eta^j_{,kl}\Re_\eta(H^k_0H^l_0)_{,i})-\etaD[\nb]K(\eta)^{-1}\tr(\eta^j_{,k}\Re_\eta H^k_{0,i})^2\\
&+ \eta_{,nm}\Re_\eta((\eta^{-1})^m_{,l} H^l_0H^n_0)+ \eta_{,n}\Re_\eta(H^n_{0,l}H^l_0).
\end{align*}
Hence, \eqref{eqn} can be written as
\begin{align}\label{eqn1}
\ddot{\eta}=&\etaD[\nb]K(\eta)^{-1}\tr(\etaD[D] \dot{\eta})^2-\etaD[\nb]K(\eta)^{-1}\tr(\eta^j_{,kl} \Re_\eta((\eta^{-1})^l_{,i}H^k_0))^2\no\\
&-\etaD[\nb]K(\eta)^{-1}\tr(\eta^j_{,kl}\Re_\eta(H^k_0H^l_0)_{,i})-\etaD[\nb]K(\eta)^{-1}\tr(\eta^j_{,k}\Re_\eta H^k_{0,i})^2\no\\
&+ \eta_{,nm}\Re_\eta((\eta^{-1})^m_{,l} H^l_0H^n_0)+ \eta_{,n}\Re_\eta(H^n_{0,l}H^l_0).
\end{align}

\section{Linearization of the equation}

We consider now a family of solutions to \eqref{eqn1} parametrized by $s$, call it $\zeta(t,s)$. Assume that $\zeta$ is differentiable in $s$ and let
$$w(t)=\D_s\zeta(t,s)|_{s=0},$$
the tangent of $\zeta(t,s)$ at $s=0$. 
Denote $\zeta(t,0)$ by $\zeta(t)$. Then $w(t)$ satisfies
\begin{align}\label{eqn2}
\ddot{w}(t)=&\D_s\ddot{\zeta}(t,s)|_{s=0}=\D_s Z(\zeta(t,s),\dot{\zeta}(t,s))|_{s=0}\no\\
=&DZ(\zeta(t,s),\dot{\zeta}(t,s))(\D_s\zeta(t,s),\D_s\dot{\zeta}(t,s))|_{s=0}\no\\
=&DZ(\zeta(t),\dot{\zeta}(t))(w(t),\dot{w}(t)).
\end{align}

In order to compute \eqref{eqn2} more explicitly,  we set $u=\Re_{\zeta^{-1}}w$ , $w=u\circ \zeta$. Clearly $K(\zeta)^{-1}$ is inverse to $K(\zeta)$ if the domain of $K(\zeta)$ is $C_0^\infty(\Omega)$. If this domain is enlarged  to $C^\infty(\Omega)$, $K(\zeta)^{-1}$ is only a right inverse.  That is $K(\zeta)K(\zeta)^{-1}=I$, but
\begin{align*}
K(\zeta)^{-1}K(\zeta)=I-\h(\zeta),
\end{align*}
where $\h(\zeta)$ is defined as follows. For $\zeta=Id$, the identity, $\h(\zeta)$ projects a function onto its harmonic part. Thus, if $\h(Id)f=g$ then $f=g$  on $\D\Omega$ and $\Delta g=0=K(Id)g$. For arbitrary $\zeta$, $\h(\zeta)f=g$ if $f=g$ on $\D\Omega$ and $K(\zeta)g=0$.

Next, we recall some identities on the commutators of operators, denoted by $[\cdot,\cdot]$, proved in \cite{Ebin}.

\begin{lemma}\label{lem.1}
	\begin{enumerate}[1)] 
		\item $\D_s\lsi[\nb]{\zeta(t,s)}|_{s=0}=\Re_{\zeta(t)}[u\cdot\nb,\nb]\Re_{\zeta(t)^{-1}}=\lsi[{[u\cdot\nb,\nb]}]{\zeta(t)}$, where $u\cdot\nb$ means the derivative in direction $u$.
		\item $\D_sK(\zeta(t,s))|_{s=0}=\lsi[{[u\cdot\nb,\Delta]}]{\zeta(t)}$.
		\item $\D_sK(\zeta(t,s))^{-1}|_{s=0}=\lsi[{([u\cdot\nb, \Delta^{-1}]-\h(u\cdot\nb)\Delta^{-1})}]{\zeta(t)}$.
	\end{enumerate}
\end{lemma}

As in \cite[(4.11)]{Ebin}, we have
\begin{align}\label{Dstr}
\D_s\tr(\lsi[D]{\zeta}\dot{\zeta})^2|_{s=0}=&2\tr\left(-(Du\circ\zeta)((D\zeta)^{-1}D\dot{\zeta})^2+(D\zeta)^{-1}D\dot{w}(D\zeta)^{-1}D\dot{\zeta}\right).
\end{align}

Since
\begin{align*}
&\D_s(\zeta^j_{,kl} \Re_\zeta((\zeta^{-1})^l_{,i}H^k_0))\\
=&\D_s\zeta^j_{,kl} \Re_\zeta((\zeta^{-1})^l_{,i}H^k_0)+\zeta^j_{,kl} \D_s(\zeta^{-1})^l_{,i}(\zeta) H^k_0(\zeta) +\zeta^j_{,kl} (\zeta^{-1})^l_{,i}(\zeta)\D_s H^k_0(\zeta)\\
=&\D_s\zeta^j_{,kl} \Re_\zeta((\zeta^{-1})^l_{,i}H^k_0)-\zeta^j_{,kl} (\zeta^{-1})^l_{,n}\D_s\zeta^n_{,m}(\zeta^{-1})^m_{,i} H^k_0(\zeta)+\zeta^j_{,kl} (\zeta^{-1})^l_{,i}  H^k_{0,m}(\zeta)\D_s\zeta^m,
\end{align*}
it follows that
\begin{align*}
&\D_s(\zeta^j_{,kl} \Re_\zeta((\zeta^{-1})^l_{,i}H^k_0))|_{s=0}\\
=&w^j_{,kl} \Re_\zeta((\zeta^{-1})^l_{,i}H^k_0)-\zeta^j_{,kl} (\zeta^{-1})^l_{,n}w^n_{,m}(\zeta^{-1})^m_{,i} H^k_0(\zeta)+\zeta^j_{,kl} (\zeta^{-1})^l_{,i}  H^k_{0,m}(\zeta)w^m,
\end{align*}
and
\begin{align*}
&\D_s\tr(\zeta^j_{,kl} \Re_\zeta((\zeta^{-1})^l_{,i}H^k_0))^2|_{s=0}\\
=&2\zeta^i_{,k'l'} \Re_\zeta((\zeta^{-1})^{l'}_{,j}H^{k'}_0)\Big(w^j_{,kl} \Re_\zeta((\zeta^{-1})^l_{,i}H^k_0)-\zeta^j_{,kl} (\zeta^{-1})^l_{,n}w^n_{,m}(\zeta^{-1})^m_{,i} H^k_0(\zeta)\\
&\qquad\qquad +\zeta^j_{,kl} (\zeta^{-1})^l_{,i}  H^k_{0,m}(\zeta)w^m \Big)\\
=&2\zeta^i_{,k'l'} \Re_\zeta((\zeta^{-1})^{l'}_{,j}H^{k'}_0)w^j_{,kl} \Re_\zeta((\zeta^{-1})^l_{,i}H^k_0)-2\zeta^i_{,k'l'}\zeta^j_{,kl} \Re_\zeta((\zeta^{-1})^{l'}_{,j}H^{k'}_0 (\zeta^{-1})^l_{,n}u^n_{,i} H^k_0)\\
&+2\zeta^i_{,k'l'} \zeta^j_{,kl}\Re_\zeta((\zeta^{-1})^{l'}_{,j}H^{k'}_0 (\zeta^{-1})^l_{,i}  H^k_{0,m}u^m).
\end{align*}

Due to
\begin{align*}
\D_s(\zeta^j_{,kl}\Re_\zeta(H^k_0H^l_0)_{,i})|_{s=0}= & \D_s\zeta^j_{,kl}\Re_\zeta(H^k_0H^l_0)_{,i}|_{s=0}+\zeta^j_{,kl}\Re_\zeta(H^k_0H^l_0)_{,im}\D_s\zeta^m|_{s=0}\\
=&w^j_{,kl}\Re_\zeta(H^k_0H^l_0)_{,i}+\zeta^j_{,kl}\Re_\zeta(H^k_0H^l_0)_{,im}w^m,
\end{align*}
we get
\begin{align*}
&\D_s\tr(\zeta^j_{,kl}\Re_\zeta(H^k_0H^l_0)_{,i})|_{s=0}
=w^i_{,kl}\Re_\zeta(H^k_0H^l_0)_{,i}+\zeta^i_{,kl}\Re_\zeta((H^k_0H^l_0)_{,im}u^m).
\end{align*}

Since
\begin{align*}
\D_s(\zeta^j_{,k}\Re_\zeta H^k_{0,i})|_{s=0}=&\D_s\zeta^j_{,k}\Re_\zeta H^k_{0,i}|_{s=0}+\zeta^j_{,k}H^k_{0,il}\D_s\zeta^l|_{s=0}\\
=&w^j_{,k}\Re_\zeta H^k_{0,i}+\zeta^j_{,k}\Re_\zeta(H^k_{0,il}u^l),
\end{align*}
it yields
\begin{align*}
\D_s\tr(\zeta^j_{,k}\Re_\zeta H^k_{0,i})^2|_{s=0}=2\zeta^i_{,k}w^j_{,k}\Re_\zeta( H^k_{0,j} H^k_{0,i})+2\zeta^i_{,k}\zeta^j_{,k}\Re_\zeta (H^k_{0,j}H^k_{0,il}u^l).
\end{align*}

Now considering \eqref{eqn1} as an equation in $\zeta(t,s)$ and applying $\D_s|_{s=0}$ to obtain
\begin{align*}
\ddot{w}=&(\D_s\lsi[\nb]{\zeta(t,s)})\lsi[\Delta]{\zeta}^{-1}\tr(\lsi[D]{\zeta} \dot{\zeta})^2+\lsi[\nb]{\zeta}\D_s(K(\zeta(t,s))^{-1})\tr(\lsi[D]{\zeta} \dot{\zeta})^2\\
&+\lsi[\nb]{\zeta}\lsi[\Delta]{\zeta}^{-1}\D_s\tr(\lsi[D]{\zeta(t,s)} \dot{\zeta}(t,s))^2-(\D_s\lsi[\nb]{\zeta(t,s)})\lsi[\Delta]{\zeta}^{-1}\tr(\zeta^j_{,kl} \Re_\zeta((\zeta^{-1})^l_{,i}H^k_0))^2\\
&-\lsi[\nb]{\zeta}\D_s(K(\zeta(t,s))^{-1})\tr(\zeta^j_{,kl} \Re_\zeta((\zeta^{-1})^l_{,i}H^k_0))^2\\
&-\lsi[\nb]{\zeta}\lsi[\Delta]{\zeta}^{-1}\D_s\tr(\zeta(t,s)^j_{,kl} \Re_{\zeta(t,s)}((\zeta(t,s)^{-1})^l_{,i}H^k_0))^2\\
&-(\D_s\lsi[\nb]{\zeta(t,s)})\lsi[\Delta]{\zeta}^{-1}\tr(\zeta^j_{,kl}\Re_\zeta(H^k_0H^l_0)_{,i})-\lsi[\nb]{\zeta}\D_s(K(\zeta(t,s))^{-1})\tr(\zeta^j_{,kl}\Re_\zeta(H^k_0H^l_0)_{,i})\\
&-\lsi[\nb]{\zeta}\lsi[\Delta]{\zeta}^{-1}\D_s\tr(\zeta(t,s)^j_{,kl}\Re_{\zeta(t,s)}(H^k_0H^l_0)_{,i})-(\D_s\lsi[\nb]{\zeta(t,s)})\lsi[\Delta]{\zeta}^{-1}\tr(\zeta^j_{,k}\Re_\zeta H^k_{0,i})^2\\
&-\lsi[\nb]{\zeta}\D_s(K(\zeta(t,s))^{-1})\tr(\zeta^j_{,k}\Re_\zeta H^k_{0,i})^2-\lsi[\nb]{\zeta}\lsi[\Delta]{\zeta}^{-1}\D_s\tr(\zeta(t,s)^j_{,k}\Re_{\zeta(t,s)} H^k_{0,i})^2\\
&+ \zeta_{,nm}\Re_\zeta((\zeta^{-1})^m_{,l} H^l_0H^n_0)+ \zeta_{,n}\Re_\zeta(H^n_{0,l}H^l_0).
\end{align*}

It yields from Lemma \ref{lem.1}
\begin{align}\label{wzeta}
\ddot{w}=&\lsi[{[u\cdot\nb,\nb]}]{\zeta}\lsi[\Delta]{\zeta}^{-1}\tr(\lsi[D]{\zeta} \dot{\zeta})^2+\lsi[\nb]{\zeta}\lsi[{([u\cdot\nb, \Delta^{-1}]-\h(u\cdot\nb)\Delta^{-1})}]{\zeta}\tr(\lsi[D]{\zeta} \dot{\zeta})^2\no\\
&+\lsi[\nb]{\zeta}\lsi[\Delta]{\zeta}^{-1}\left(2\tr\left(-(Du\circ\zeta)((D\zeta)^{-1}D\dot{\zeta})^2+(D\zeta)^{-1}D\dot{w}(D\zeta)^{-1}D\dot{\zeta}\right)\right)\no\\
&-\lsi[{[u\cdot\nb,\nb]}]{\zeta}\lsi[\Delta]{\zeta}^{-1}\tr(\zeta^j_{,kl} \Re_\zeta((\zeta^{-1})^l_{,i}H^k_0))^2\no\\
&-\lsi[\nb]{\zeta}\lsi[{([u\cdot\nb, \Delta^{-1}]-\h(u\cdot\nb)\Delta^{-1})}]{\zeta}\tr(\zeta^j_{,kl} \Re_\zeta((\zeta^{-1})^l_{,i}H^k_0))^2\no\\
&-\lsi[\nb]{\zeta}\lsi[\Delta]{\zeta}^{-1}(2\zeta^i_{,k'l'} \Re_\zeta((\zeta^{-1})^{l'}_{,j}H^{k'}_0)w^j_{,kl} \Re_\zeta((\zeta^{-1})^l_{,i}H^k_0))\no\\
&+\lsi[\nb]{\zeta}\lsi[\Delta]{\zeta}^{-1}(2\zeta^i_{,k'l'}\zeta^j_{,kl} \Re_\zeta((\zeta^{-1})^{l'}_{,j}H^{k'}_0 (\zeta^{-1})^l_{,n}u^n_{,i} H^k_0))\no\\
&-\lsi[\nb]{\zeta}\lsi[\Delta]{\zeta}^{-1}(2\zeta^i_{,k'l'} \zeta^j_{,kl}\Re_\zeta((\zeta^{-1})^{l'}_{,j}H^{k'}_0 (\zeta^{-1})^l_{,i}  H^k_{0,m}u^m))\no\\
&-\lsi[{[u\cdot\nb,\nb]}]{\zeta}\lsi[\Delta]{\zeta}^{-1}(\zeta^j_{,kl}\Re_\zeta(H^k_0H^l_0)_{,j})\no\\
&-\lsi[\nb]{\zeta}\lsi[{([u\cdot\nb, \Delta^{-1}]-\h(u\cdot\nb)\Delta^{-1})}]{\zeta}(\zeta^j_{,kl}\Re_\zeta(H^k_0H^l_0)_{,j})\no\\
&-\lsi[\nb]{\zeta}\lsi[\Delta]{\zeta}^{-1}(w^i_{,kl}\Re_\zeta(H^k_0H^l_0)_{,i})-\lsi[\nb]{\zeta}\lsi[\Delta]{\zeta}^{-1}(\zeta^i_{,kl}\Re_\zeta((H^k_0H^l_0)_{,im}u^m))\no\\
&-\lsi[{[u\cdot\nb,\nb]}]{\zeta}\lsi[\Delta]{\zeta}^{-1}\tr(\zeta^j_{,k}\Re_\zeta H^k_{0,i})^2\no\\
&-\lsi[\nb]{\zeta}\lsi[{([u\cdot\nb, \Delta^{-1}]-\h(u\cdot\nb)\Delta^{-1})}]{\zeta}\tr(\zeta^j_{,k}\Re_\zeta H^k_{0,i})^2\no\\
&-\lsi[\nb]{\zeta}\lsi[\Delta]{\zeta}^{-1}(2\zeta^i_{,k'}w^j_{,k}\Re_\zeta( H^{k'}_{0,j} H^k_{0,i}))-\lsi[\nb]{\zeta}\lsi[\Delta]{\zeta}^{-1}(2\zeta^i_{,k'}\zeta^j_{,k}\Re_\zeta (H^{k'}_{0,j}H^k_{0,il}u^l))\no\\
&+ \zeta_{,nm}\Re_\zeta((\zeta^{-1})^m_{,l} H^l_0H^n_0)+ \zeta_{,n}\Re_\zeta(H^n_{0,l}H^l_0)\no\\
=&\lsi[{[u\cdot\nb,\nb\Delta^{-1}]}]{\zeta}\tr(\lsi[D]{\zeta} \dot{\zeta})^2-\lsi[\nb]{\zeta}\lsi[{(\h(u\cdot\nb)\Delta^{-1})}]{\zeta}\tr(\lsi[D]{\zeta} \dot{\zeta})^2\no\\
&+\lsi[\nb]{\zeta}\lsi[\Delta]{\zeta}^{-1}\left(2\tr\left(-(Du\circ\zeta)((D\zeta)^{-1}D\dot{\zeta})^2+(D\zeta)^{-1}D\dot{w}(D\zeta)^{-1}D\dot{\zeta}\right)\right)\no\\
&-\lsi[{[u\cdot\nb,\nb\Delta^{-1}]}]{\zeta}\tr(\zeta^j_{,kl} \Re_\zeta((\zeta^{-1})^l_{,i}H^k_0))^2\no\\
&+\lsi[\nb]{\zeta}\lsi[{(\h(u\cdot\nb)\Delta^{-1})}]{\zeta}\tr(\zeta^j_{,kl} \Re_\zeta((\zeta^{-1})^l_{,i}H^k_0))^2\no\\
&-\lsi[\nb]{\zeta}\lsi[\Delta]{\zeta}^{-1}(2\zeta^i_{,k'l'} \Re_\zeta((\zeta^{-1})^{l'}_{,j}H^{k'}_0)w^j_{,kl} \Re_\zeta((\zeta^{-1})^l_{,i}H^k_0))\no\\
&+\lsi[\nb]{\zeta}\lsi[\Delta]{\zeta}^{-1}(2\zeta^i_{,k'l'}\zeta^j_{,kl} \Re_\zeta((\zeta^{-1})^{l'}_{,j}H^{k'}_0 (\zeta^{-1})^l_{,n}u^n_{,i} H^k_0))\no\\
&-\lsi[\nb]{\zeta}\lsi[\Delta]{\zeta}^{-1}(2\zeta^i_{,k'l'} \zeta^j_{,kl}\Re_\zeta((\zeta^{-1})^{l'}_{,j}H^{k'}_0 (\zeta^{-1})^l_{,i}  H^k_{0,m}u^m))\no\\
&-\lsi[{[u\cdot\nb,\nb\Delta^{-1}]}]{\zeta}(\zeta^j_{,kl}\Re_\zeta(H^k_0H^l_0)_{,j})+\lsi[\nb]{\zeta}\lsi[{(\h(u\cdot\nb)\Delta^{-1})}]{\zeta}(\zeta^j_{,kl}\Re_\zeta(H^k_0H^l_0)_{,j})\no\\
&-\lsi[\nb]{\zeta}\lsi[\Delta]{\zeta}^{-1}(w^i_{,kl}\Re_\zeta(H^k_0H^l_0)_{,i})-\lsi[\nb]{\zeta}\lsi[\Delta]{\zeta}^{-1}(\zeta^i_{,kl}\Re_\zeta((H^k_0H^l_0)_{,im}u^m))\no\\
&-\lsi[{[u\cdot\nb,\nb\Delta^{-1}]}]{\zeta}\tr(\zeta^j_{,k}\Re_\zeta H^k_{0,i})^2+\lsi[\nb]{\zeta}\lsi[{(\h(u\cdot\nb)\Delta^{-1})}]{\zeta}\tr(\zeta^j_{,k}\Re_\zeta H^k_{0,i})^2\no\\
&-\lsi[\nb]{\zeta}\lsi[\Delta]{\zeta}^{-1}(2\zeta^i_{,k'}w^j_{,k}\Re_\zeta( H^{k'}_{0,j} H^k_{0,i}))
-\lsi[\nb]{\zeta}\lsi[\Delta]{\zeta}^{-1}(2\zeta^i_{,k'}\zeta^j_{,k}\Re_\zeta (H^{k'}_{0,j}H^k_{0,il}u^l))\no\\
&+ \zeta_{,nm}\Re_\zeta((\zeta^{-1})^m_{,l} H^l_0H^n_0)+ \zeta_{,n}\Re_\zeta(H^n_{0,l}H^l_0).
\end{align}

For our choice of $\eta(t)$ in place of $\zeta(t)$ we can compute explicitly some  terms in the linearized equation.  To do so, let us return to complex variables for convenience.  We know that $D\eta(t,z)=e^{\im t}$ for $\eta$ given in  section \ref{sec.examp} so that $(D\eta)^{-1}D\dot{\eta}=\im$ and $\tr(-(Du\circ\eta)(-1))=\dv u=0$. Also,  $\tr(\etaD[D] \dot{\eta})^2=\tr((D \vv)^2)\circ\eta=-2$. Due to $\eta^i_{,j}=e^{\im t}\delta^i_j$, $(\eta^{-1})^i_{,j}=e^{-\im t}\delta^i_j$ and $H^n_{0,l}=(l-n)b$ ($l,n\in\{1,2\}$). Therefore,  any second-order spatial derivatives of both $\eta$, $\eta^{-1}$ and $H_0$ are zeros. Hence $\tr(\eta^j_{,k}\Re_\eta H^k_{0,i})^2=-2b^2e^{2\im t}$ and $\eta_{,n}\Re_\eta(H^n_{0,l}H^l_0)=b^2e^{\im t}\eta$. Therefore, it turns out from \eqref{wzeta}
\begin{subequations}\label{weta}
	\begin{align}
	\ddot{w}=&(1-b^2)\big(\lsi[{[u\cdot\nb,\nb\Delta^{-1}]}]{\eta}(-2)-\lsi[\nb]{\eta}\lsi[{(\h(u\cdot\nb)\Delta^{-1})}]{\eta}(-2)\big)\label{weta.1}\\
	&+\lsi[\nb]{\eta}\lsi[\Delta]{\eta}^{-1}\left(2\tr\left(Du\circ\eta+\im e^{-\im t}D\dot{w}\right)\right)\label{weta.2}\\
	&-\lsi[\nb]{\eta}\lsi[\Delta]{\eta}^{-1}(w^i_{,kl}\Re_\eta(H^k_0H^l_0)_{,i})\label{weta.3}\\
	&-e^{\im t}\lsi[\nb]{\eta}\lsi[\Delta]{\eta}^{-1}(2w^j_{,k}\Re_\eta( H^{i}_{0,j} H^k_{0,i}))\label{weta.4}\\
	&+b^2 e^{\im t}\eta.\label{weta.5}
	\end{align}
\end{subequations}

Now we restrict our attention to special solutions of \eqref{weta}. Assume that $w=\nb f(t)\circ\eta$ where $f(t)$ is a harmonic function on $\Omega$. Thus, $u=\nb f$ a harmonic gradient. Then $Dw=D\eta(D\nb f)\circ \eta$ and $\tr(\im e^{-\im t}D\dot{w})=0$ as similar as in \cite[Eq. (4.16)]{Ebin}, so \eqref{weta.2} vanishes. We also have
\begin{align*}
w^j_{,k}=&\D_l\D_jf(t,\eta(t,x))\eta^l_{,k}(t,x)=e^{\im t}\D_k\D_jf(t,\eta(t,x)),\\
w^i_{,kl}=&e^{\im t}\D_m\D_k\D_if(t,\eta(t,x))\eta^m_{,l}=e^{2\im t}\D_l\D_k\D_if(t,\eta(t,x)),
\end{align*}
and then
\begin{align}
w^i_{,kl}\Re_\eta(H^k_0H^l_0)_{,i}=&2b^2e^{2\im t}\sum_{i,k=1}^2\D_l\D_k\D_if(t,\eta(t,x))((i-k)\im \eta^l)\no\\
=&2b^2\im e^{3\im t}\sum_{i,k=1}^2x^l\D_l\D_k\D_if(t,\eta(t,x))(i-k)\no\\
=&0,\label{wH.1}\\
\intertext{and}
w^j_{,k}\Re_\eta( H^{i}_{0,j} H^k_{0,i})=&b^2e^{\im t}\sum_{i,j,k=1}^2\D_k\D_jf(t,\eta(t,x))(j-i)(i-k)\no\\
=&-b^2e^{\im t}\sum_{i,k=1}^2\D_k^2f(t,\eta(t,x))(i-k)^2\no\\
=&-b^2e^{\im t}(\D_{2}^2f(t,\eta(t,x))+\D_{1}^2f(t,\eta(t,x)))\no\\
=&-b^2e^{\im t}\Delta f\circ\eta=0.\label{wH.2}
\end{align}
Thus, both \eqref{weta.3} and \eqref{weta.4} vanish. Hence,  similar to the discussion  in \cite{Ebin}, 
\begin{align}\label{eqf}
(\nb f\circ\eta)^{\cdot\cdot}=&(1-b^2)x\cdot\nb(\nb f\circ\eta)+ b^2 e^{\im t}\eta,
\end{align}
where $x\cdot\nb$ is the radial derivative and commutes with composition with the rotation $\eta$. Now, we recall a useful lemma.

\begin{lemma}[\mbox{\cite[Lemma 4.19]{Ebin}}] \label{lem.4.19}
	If $f$ is harmonic, and $\eta(t,z)=e^{\im t}z$, then there exists a harmonic function $g$ such that $\nb f\circ \eta=\nb g$.
\end{lemma}

Since $f$ is harmonic,  by Lemma \ref{lem.4.19},  there exists a harmonic function $g$ such that $\nb f\circ \eta=\nb g$. Thus, \eqref{eqf} can be written as
\begin{align}\label{eqg}
(\nb g)^{\cdot\cdot}=(1-b^2)x\cdot\nb(\nb g)+b^2e^{2\im t}z.
\end{align}

Define $\A$ by $\A\nb g=(x\cdot\nb)\nb g$. Let $g(z)=\RE z^n$ ($n\gs 1$), $\nb g=n\bar{z}^{n-1}$. Then $\A\nb g=nz\cdot\nb \bar{z}^{n-1}=n(n-1)\bar{z}^{n-1}=(n-1)\nb g$. Similarly, if $g(z)=\RE \im z^n$ ($n\gs 1$), $\nb g=-\im n\bar{z}^{n-1}$. Then $\A\nb g=(n-1)\nb g$. Note that $E:=\{n\bar{z}^{n-1},\im n\bar{z}^{n-1}\}_{n=1}^\infty$ form a basis of the set of Harmonic gradients on $\Omega$. So $\A$ has this set as a complete set of eigenfunctions and has double eigenvalues $0,1,2,\cdots$. We seek the solutions by  separated variables of the form $\nb g(t,z)=\sigma(t)h(z)$ for $h\in E$, so that \eqref{eqg} can be written as
\begin{align}\label{eqt}
\ddot{\sigma}(t)=(1-b^2)(n-1)\sigma(t)+b^2e^{2\im t}z/h(z).
\end{align}
Let $B=\sqrt{(1-b^2)(n-1)}$. The usual solution to \eqref{eqt} can be written as 
\begin{align*}
\sigma(t)=&C_1e^{Bt}+C_2e^{-Bt}+\int_0^t\frac{e^{Bs}e^{-Bt}-e^{Bt}e^{-Bs}}{-Be^{Bs}e^{-Bs}-Be^{Bs}e^{-Bs}}b^2e^{2\im s}z/h(z)ds\\
=&C_1e^{Bt}+C_2e^{-Bt}-\frac{b^2z}{2Bh(z)}\int_0^t(e^{Bs}e^{-Bt}-e^{Bt}e^{-Bs})e^{2\im s}ds\\
=&C_1e^{Bt}+C_2e^{-Bt}-\frac{b^2z}{2Bh(z)}\left(\frac{e^{2\im t}-e^{-Bt}}{B+2\im}+\frac{e^{2\im t}-e^{Bt}}{B-2\im} \right)\\
=&C_1e^{Bt}+C_2e^{-Bt}-\frac{b^2z}{2Bh(z)}\left(\frac{2B}{B^2+4}e^{2\im t}-\frac{e^{-Bt}}{B+2\im}-\frac{e^{Bt}}{B-2\im} \right).
\end{align*}
Hence, we get the solution of \eqref{eqg} of the form
\begin{align*}
w_n(t,z)=C_1e^{Bt}n\bar{z}^{n-1}+C_2e^{-Bt}n\bar{z}^{n-1}-\frac{b^2z}{2B}\left(\frac{2B}{B^2+4}e^{2\im t}-\frac{e^{-Bt}}{B+2\im}-\frac{e^{Bt}}{B-2\im} \right).
\end{align*}
If we assume the initial conditions $w_n(0)=0$ and $\dot{w}_n(0)=e^{-n^{1/4}}\bar{z}^n$, then 
\begin{align*}
C_1+C_2=0, \text{ and } (C_1-C_2)Bn\bar{z}^{n-1}=e^{-n^{1/4}}\bar{z}^n,
\end{align*}
i.e., $C_1=-C_2=e^{-n^{1/4}}\bar{z}/(2Bn)$. Hence, for $n\gs 2$,
\begin{align}\label{soln}
w_n(t)=&\frac{1}{\sqrt{(1-b^2)(n-1)}}e^{-n^{1/4}}\sinh(\sqrt{(1-b^2)(n-1)}t)\bar{z}^{n}\no\\
&+\frac{b^2z}{(1-b^2)(n-1)+4}\left(\cosh(\sqrt{(1-b^2)(n-1)}t)-e^{2\im t}\right)\no\\
&+\frac{2b^2\im z\sinh(\sqrt{(1-b^2)(n-1)}t)}{\sqrt{(1-b^2)(n-1)}((1-b^2)(n-1)+4)},
\end{align}
is a sequence of solutions to \eqref{weta} with initial conditions $w_n(0)=0$ and $\dot{w}_n(0)=e^{-n^{1/4}}\bar{z}^n$. When $b^2<1$, this sequence is as useful as for the Euler equation, discussed in \cite{Ebin}, because the initial data go to zero in $C^\infty(\Omega)$, but for any $t>0$, in view of  the exponential growth of $\sinh$ and $\cosh$, 
$\{w_n(t)\}_{n=2}^\infty$ is an unbounded sequence in $C^\infty(\Omega)$. 

\section{Discontinuous dependence on initial data}

In this section, we will show  that the problem \eqref{eq.eta} is ill-posed. 

\subsection{Construction of the sequences of initial data and solutions}

Let $\eta(t,z)=e^{\im t}z$, the solution to \eqref{eq.eta} given in section \ref{sec.examp}, and set $\zeta_n(0,z)=\eta(0,z)=z$ and $\dot{\zeta}_n(0,z)=\dot{\eta}(0,z)+e^{-n^{1/4}}\bar{z}^n$. Then, $(\zeta_n(0,z),\dot{\zeta}_n(0,z))\to (\eta(0,z),\dot{\eta}(0,z))$ in $C^\infty(\Omega)\times C^\infty(\Omega)$ as $n\to\infty$.  Assuming that there exists some positive $T$ such that for all $n$, $\zeta_n(t)$ is the unique solution of \eqref{eq.eta} for $0\ls t\ls T$, the goal is to show that 
$\zeta_n(t)$ does not converge to $\eta(t)$, not in $C^\infty(\Omega)$ for any positive $t\ls T$.

For this,  we define
\begin{align*}
y_n(t)=\zeta_n(t)-\eta(t).
\end{align*}
We will show that $y_n(t)$ grows like $w_n(t)$ of \eqref{soln} and thus, $\lim_{n\to\infty}(\zeta_n(t)-\eta(t))$ is not zero for $t>0$.

Since both $\eta(t)$ and $\zeta_n(t)$ satisfy \eqref{eqn}, we have
\begin{align*}
\ddot{\zeta}-\ddot{\eta}=&Z(\zeta,\dot{\zeta})-Z(\eta,\dot{\eta})=\int_0^1 Z_{,j}(\zeta(s),\dot{\zeta}(s))(\zeta-\eta,\dot{\zeta}-\dot{\eta})^jds\no\\
=:&\int_0^1 DZ(\zeta(s),\dot{\zeta}(s))(\zeta-\eta,\dot{\zeta}-\dot{\eta})ds,
\end{align*}
where $\zeta(s)=\zeta+s(\eta-\zeta)$.  Hereafter we suppress the subscript ``$n$'' in $\zeta$ for simplicity. Thus, by the mean value theorem of integrals,
\begin{align}
\ddot{y}(t)=&\int_0^1 Z_{,j}(\zeta(s),\dot{\zeta}(s))(y,\dot{y})^jds\no\\
=&Z_{,j}(\eta,\dot{\eta})(y,\dot{y})^j+\int_0^1(Z_{,j}(\zeta(s),\dot{\zeta}(s))(y,\dot{y})^j-Z_{,j}(\eta,\dot{\eta})(y,\dot{y})^j)ds\no\\
=&Z_{,j}(\eta,\dot{\eta})(y,\dot{y})^j+\int_0^1(1-s)\left(\int_0^s Z_{,jk}(\zeta(\sigma),\dot{\zeta}(\sigma))(y,\dot{y})^kd\sigma\right)(y,\dot{y})^jds\no\\
=:&DZ(\eta,\dot{\eta})(y,\dot{y})+\int_0^1(1-s)\left(\int_0^s D^2Z(\zeta(\sigma),\dot{\zeta}(\sigma))((y,\dot{y}),(y,\dot{y}))d\sigma\right)ds,\label{y.eqn}
\end{align}
since $\zeta(s)-\eta=\zeta+s(\eta-\zeta)-\eta=(1-s)(\zeta-\eta)=(1-s)y$ for $s\in [0,1]$ and $(1-\sigma')\zeta(s)+\sigma'\eta=(1-\sigma')((1-s)\zeta+s\eta)+\sigma'\eta=(1-s)(1-\sigma')\zeta+(s+(1-s)\sigma')\eta=(1-\sigma)\zeta+\sigma\eta=\zeta(\sigma)$, where $\sigma=s+(1-s)\sigma'$ and $1-\sigma=(1-s)(1-\sigma')$ for $\sigma'\in[0,1]$ and $\sigma\in[0,s]$.

Denote $H^s(\Omega)$ the Sobolev spaces  for any real number $s$ with inner product $(\cdot,\cdot)_s$ and corresponding norm $\norm{\cdot}_s$. For a nonnegative integer $s$,
\begin{align*}
(f,g)_s=\sum_{k=0}^s\int_{\Omega} \langle D^kf,D^k g\rangle dx,
\end{align*}
with the usual extensions to other $s$.

\begin{proposition}\label{prop.1}
	Let $s\gs 1$ and $H_0\in H^{s+2}$. Then, 
	\begin{align*}
	\norm{D^2Z(\zeta,\dot{\zeta})((y,\dot{y}),(y,\dot{y}))}_s
	\ls& C\norm{(y,\dot{y})}_{1}^{(2s-1)/2s}\norm{D^{s+1}(y,\dot{y})}_{0}^{(2s+1)/2s}\\
	&+ C\norm{(y,\dot{y})}_{s}\norm{D^{s+1}(y,\dot{y})}_{0}+Cb^2\norm{y}_{s+2},
	\end{align*} 
	where $C$ is uniform for all $(\zeta,\dot{\zeta})$ in a $H^{s+3}$ neighborhood of the curve $(\eta(t),\dot{\eta}(t))$.
\end{proposition}

\begin{proof}
	
	We compute $DZ(\eta,\dot{\eta})(y,\dot{y})$  first. Let  $y=\D_s\zeta(t,s)|_{s=0}$ and $u=y\circ\zeta^{-1}$  for convenience. Since $(D\zeta)^{-1}=D(\zeta^{-1})\circ\zeta$, so $$(D\zeta)^{-1}D\dot{\zeta}=D(\zeta^{-1})^j\circ\zeta D_j\dot{\zeta}=\Re_\zeta (D(\zeta^{-1})^j (D_j\dot{\zeta}\circ\zeta^{-1}))=\Re_\zeta D(\dot{\zeta}\circ\zeta^{-1})=D_\zeta \dot{\zeta}$$ and similarly $(D\zeta)^{-1}D\dot{y}=D_\zeta \dot{y}$. Moreover, 
	\begin{align*}
	[u\cdot\nb,D]_\zeta \dot{\zeta}=&\Re_{\zeta}[u\cdot\nb,D] (\dot{\zeta}\circ\zeta^{-1})=\Re_{\zeta}(u\cdot\nb D (\dot{\zeta}\circ\zeta^{-1}))-\Re_{\zeta } (D(u\cdot\nb (\dot{\zeta}\circ\zeta^{-1})))\\
	=& -\Re_{\zeta} (Du\cdot\nb (\dot{\zeta}\circ\zeta^{-1}))
	= -\Re_{\zeta} (Du \Re_{\zeta^{-1}} D_\zeta \dot{\zeta})\\
	=& -(Du\circ\zeta) D_\zeta \dot{\zeta}.
	\end{align*}
	Thus,  similar to \eqref{wzeta}, one has
	\begin{align}
	DZ(\zeta,\dot{\zeta})(y,\dot{y})=&\lsi[{\left([u\cdot\nb,\nb\Delta^{-1}]-\nb(\h(u\cdot\nb)\Delta^{-1})\right)}]{\zeta}\tr(\lsi[D]{\zeta} \dot{\zeta})^2\no\\
	&+2\lsi[\nb]{\zeta}\lsi[\Delta]{\zeta}^{-1}\tr\left(D_\zeta\dot{\zeta}D_\zeta\dot{y}+D_\zeta\dot{\zeta}([u\cdot\nb,D]_\zeta \dot{\zeta})\right)\no\\
	&-\lsi[{[u\cdot\nb,\nb\Delta^{-1}]}]{\zeta}\tr(\zeta^j_{,kl} \Re_\zeta((\zeta^{-1})^l_{,i}H^k_0))^2\no\\
	&+\lsi[\nb]{\zeta}\lsi[{(\h(u\cdot\nb)\Delta^{-1})}]{\zeta}\tr(\zeta^j_{,kl} \Re_\zeta((\zeta^{-1})^l_{,i}H^k_0))^2\no\\
	&-\lsi[\nb]{\zeta}\lsi[\Delta]{\zeta}^{-1}(2\zeta^i_{,k'l'} \Re_\zeta((\zeta^{-1})^{l'}_{,j}H^{k'}_0)y^j_{,kl} \Re_\zeta((\zeta^{-1})^l_{,i}H^k_0))\no\\
	&+\lsi[\nb]{\zeta}\lsi[\Delta]{\zeta}^{-1}(2\zeta^i_{,k'l'}\zeta^j_{,kl} \Re_\zeta((\zeta^{-1})^{l'}_{,j}H^{k'}_0 (\zeta^{-1})^l_{,n}u^n_{,i} H^k_0))\no\\
	&-\lsi[\nb]{\zeta}\lsi[\Delta]{\zeta}^{-1}(2\zeta^i_{,k'l'} \zeta^j_{,kl}\Re_\zeta((\zeta^{-1})^{l'}_{,j}H^{k'}_0 (\zeta^{-1})^l_{,i}  H^k_{0,m}u^m))\no\\
	&-\lsi[{[u\cdot\nb,\nb\Delta^{-1}]}]{\zeta}(\zeta^j_{,kl}\Re_\zeta(H^k_0H^l_0)_{,j})\no\\
	&+\lsi[\nb]{\zeta}\lsi[{(\h(u\cdot\nb)\Delta^{-1})}]{\zeta}(\zeta^j_{,kl}\Re_\zeta(H^k_0H^l_0)_{,j})\no\\
	&-\lsi[\nb]{\zeta}\lsi[\Delta]{\zeta}^{-1}(y^i_{,kl}\Re_\zeta(H^k_0H^l_0)_{,i})-\lsi[\nb]{\zeta}\lsi[\Delta]{\zeta}^{-1}(\zeta^i_{,kl}\Re_\zeta((H^k_0H^l_0)_{,im}u^m))\no\\
	&-\lsi[{[u\cdot\nb,\nb\Delta^{-1}]}]{\zeta}\tr(\zeta^j_{,k}\Re_\zeta H^k_{0,i})^2\no\\
	&+\lsi[\nb]{\zeta}\lsi[{(\h(u\cdot\nb)\Delta^{-1})}]{\zeta}\tr(\zeta^j_{,k}\Re_\zeta H^k_{0,i})^2\no\\
	&-\lsi[\nb]{\zeta}\lsi[\Delta]{\zeta}^{-1}(2\zeta^i_{,k'}y^j_{,k}\Re_\zeta( H^{k'}_{0,j} H^k_{0,i}))\no\\
	&-\lsi[\nb]{\zeta}\lsi[\Delta]{\zeta}^{-1}(2\zeta^i_{,k'}\zeta^j_{,k}\Re_\zeta (H^{k'}_{0,j}H^k_{0,il}u^l))\no\\
	&+ \zeta_{,nm}\Re_\zeta((\zeta^{-1})^m_{,l} H^l_0H^n_0)+ \zeta_{,n}\Re_\zeta(H^n_{0,l}H^l_0). \label{DZ}
	\end{align}
	
	For any function $f$ on $\Omega$,  $\lsi[\nb^j]{\zeta}f=\Re_{\zeta}\D_j(f(\zeta^{-1}))=\Re_{\zeta}\D_if(\zeta^{-1})\D_j(\zeta^{-1})^i=(D(\zeta^{-1})\circ\zeta\nb f)^j$ and $\frac{\D}{\D \zeta^j}f=\frac{\D}{\D \zeta^j}f(\zeta^{-1}(\zeta))=\D_if \frac{\D}{\D \zeta^j}(\zeta^{-1}(\zeta))^i=(D(\zeta^{-1})\circ\zeta\nb f)^j$, i.e., $\lsi[\nb^j]{\zeta}f\equiv\frac{\D}{\D \zeta^j}f$. Thus,
	\begin{align*}
	D^2Z(\zeta,\dot{\zeta})((y,\dot{y}),(y,\dot{y}))=\lsi[\nb^j]{\zeta}DZ(\zeta,\dot{\zeta})(y,\dot{y})y_j+\frac{\D}{\D \dot{\zeta}^j}DZ(\zeta,\dot{\zeta})(y,\dot{y})\dot{y}_j.
	\end{align*}
	Clearly, $$\frac{\D}{\D \dot{\zeta}^j}DZ(\zeta,\dot{\zeta})(y,\dot{y})\dot{y}_j=0$$
	due to the form of $DZ(\zeta,\dot{\zeta})(y,\dot{y})$ in \eqref{DZ} where all $\dot{\zeta}$'s appear in the form of derivatives like $\lsi[D]{\zeta}\dot{\zeta}$. So only the first term needs to be considered. 
	\begin{subequations}\label{D2Z}
		\begin{align}
		y_j\lsi[\nb^j]{\zeta}D&Z(\zeta,\dot{\zeta})(y,\dot{y})=D_yDZ(\zeta,\dot{\zeta})(y,\dot{y})\no\\
		=&D_y\lsi[{\left([u\cdot\nb,\nb\Delta^{-1}]-\nb(\h(u\cdot\nb)\Delta^{-1})\right)}]{\zeta}\tr(\lsi[D]{\zeta} \dot{\zeta})^2 \label{D2Z.1}\\
		&+\lsi[{\left([u\cdot\nb,\nb\Delta^{-1}]-\nb(\h(u\cdot\nb)\Delta^{-1})\right)}]{\zeta}D_y\tr(\lsi[D]{\zeta} \dot{\zeta})^2 \label{D2Z.1'}\\
		&+2D_y(\lsi[\nb]{\zeta}\lsi[\Delta]{\zeta}^{-1})\tr\left(D_\zeta\dot{\zeta}D_\zeta\dot{y}+D_\zeta\dot{\zeta}([u\cdot\nb,D]_\zeta \dot{\zeta})\right)\label{D2Z.2}\\
		&+2\lsi[\nb]{\zeta}\lsi[\Delta]{\zeta}^{-1}D_y\tr\left(D_\zeta\dot{\zeta}D_\zeta\dot{y}+D_\zeta\dot{\zeta}([u\cdot\nb,D]_\zeta \dot{\zeta})\right)\label{D2Z.2'}\\
		&-D_y\lsi[{\left([u\cdot\nb,\nb\Delta^{-1}]-\nb(\h(u\cdot\nb)\Delta^{-1})\right)}]{\zeta}\tr(\zeta^m_{,kl} \Re_\zeta((\zeta^{-1})^l_{,i}H^k_0))^2\label{D2Z.3}\\
		&-\lsi[{\left([u\cdot\nb,\nb\Delta^{-1}]-\nb(\h(u\cdot\nb)\Delta^{-1})\right)}]{\zeta}D_y\tr(\zeta^m_{,kl} \Re_\zeta((\zeta^{-1})^l_{,i}H^k_0))^2 \label{D2Z.4}\\
		&-D_y(\lsi[\nb]{\zeta}\lsi[\Delta]{\zeta}^{-1})(2\zeta^i_{,k'l'} \Re_\zeta((\zeta^{-1})^{l'}_{,m}H^{k'}_0)y^m_{,kl} \Re_\zeta((\zeta^{-1})^l_{,i}H^k_0))\label{D2Z.5}\\
		&-\lsi[\nb]{\zeta}\lsi[\Delta]{\zeta}^{-1}D_y(2\zeta^i_{,k'l'} \Re_\zeta((\zeta^{-1})^{l'}_{,m}H^{k'}_0)y^m_{,kl} \Re_\zeta((\zeta^{-1})^l_{,i}H^k_0))\label{D2Z.5'}\\
		&+D_y(\lsi[\nb]{\zeta}\lsi[\Delta]{\zeta}^{-1})(2\zeta^i_{,k'l'}\zeta^m_{,kl} \Re_\zeta((\zeta^{-1})^{l'}_{,m}H^{k'}_0 (\zeta^{-1})^l_{,n}u^n_{,i} H^k_0))\label{D2Z.6}\\
		&+\lsi[\nb]{\zeta}\lsi[\Delta]{\zeta}^{-1}D_y(2\zeta^i_{,k'l'}\zeta^m_{,kl} \Re_\zeta((\zeta^{-1})^{l'}_{,m}H^{k'}_0 (\zeta^{-1})^l_{,n}u^n_{,i} H^k_0))\label{D2Z.6'}\\
		&-D_y(\lsi[\nb]{\zeta}\lsi[\Delta]{\zeta}^{-1})(2\zeta^i_{,k'l'} \zeta^{j'}_{,kl}\Re_\zeta((\zeta^{-1})^{l'}_{,j'}H^{k'}_0 (\zeta^{-1})^l_{,i}  H^k_{0,m}u^m))\label{D2Z.7}\\
		&-\lsi[\nb]{\zeta}\lsi[\Delta]{\zeta}^{-1}D_y(2\zeta^i_{,k'l'} \zeta^{j'}_{,kl}\Re_\zeta((\zeta^{-1})^{l'}_{,j'}H^{k'}_0 (\zeta^{-1})^l_{,i}  H^k_{0,m}u^m))\label{D2Z.7'}\\
		&-D_y\lsi[{\left([u\cdot\nb,\nb\Delta^{-1}]-\nb(\h(u\cdot\nb)\Delta^{-1})\right)}]{\zeta}(\zeta^m_{,kl}\Re_\zeta(H^k_0H^l_0)_{,m})\label{D2Z.8}\\
		&-\lsi[{\left([u\cdot\nb,\nb\Delta^{-1}]-\nb(\h(u\cdot\nb)\Delta^{-1})\right)}]{\zeta}D_y(\zeta^m_{,kl}\Re_\zeta(H^k_0H^l_0)_{,m})\label{D2Z.9}\\
		&-D_y(\lsi[\nb]{\zeta}\lsi[\Delta]{\zeta}^{-1})(y^i_{,kl}\Re_\zeta(H^k_0H^l_0)_{,i})\label{D2Z.10}\\
		&-\lsi[\nb]{\zeta}\lsi[\Delta]{\zeta}^{-1}D_y(y^i_{,kl}\Re_\zeta(H^k_0H^l_0)_{,i})\label{D2Z.10'}\\
		&-D_y(\lsi[\nb]{\zeta}\lsi[\Delta]{\zeta}^{-1})(\zeta^i_{,kl}\Re_\zeta((H^k_0H^l_0)_{,im}u^m))\label{D2Z.11}\\
		&-\lsi[\nb]{\zeta}\lsi[\Delta]{\zeta}^{-1}D_y(\zeta^i_{,kl}\Re_\zeta((H^k_0H^l_0)_{,im}u^m))\label{D2Z.11'}\\
		&-D_y\lsi[{\left([u\cdot\nb,\nb\Delta^{-1}]-\nb(\h(u\cdot\nb)\Delta^{-1})\right)}]{\zeta}\tr(\zeta^m_{,k}\Re_\zeta H^k_{0,i})^2\label{D2Z.12}\\
		&-\lsi[{\left([u\cdot\nb,\nb\Delta^{-1}]-\nb(\h(u\cdot\nb)\Delta^{-1})\right)}]{\zeta}D_y\tr(\zeta^m_{,k}\Re_\zeta H^k_{0,i})^2\label{D2Z.13}\\
		&-D_y(\lsi[\nb]{\zeta}\lsi[\Delta]{\zeta}^{-1})(2\zeta^i_{,k'}y^m_{,k}\Re_\zeta( H^{k'}_{0,m} H^k_{0,i}))\label{D2Z.14}\\
		&-\lsi[\nb]{\zeta}\lsi[\Delta]{\zeta}^{-1}D_y(2\zeta^i_{,k'}y^m_{,k}\Re_\zeta( H^{k'}_{0,m} H^k_{0,i}))\label{D2Z.14'}\\
		&-D_y(\lsi[\nb]{\zeta}\lsi[\Delta]{\zeta}^{-1})(2\zeta^i_{,k'}\zeta^m_{,k}\Re_\zeta (H^{k'}_{0,m}H^k_{0,il}u^l))\label{D2Z.15}\\
		&-\lsi[\nb]{\zeta}\lsi[\Delta]{\zeta}^{-1}D_y(2\zeta^i_{,k'}\zeta^m_{,k}\Re_\zeta (H^{k'}_{0,m}H^k_{0,il}u^l))\label{D2Z.15'}\\
		&+D_y(\zeta_{,nm}\Re_\zeta((\zeta^{-1})^m_{,l} H^l_0H^n_0)+ \zeta_{,n}\Re_\zeta(H^n_{0,l}H^l_0)).\label{D2Z.16}
		\end{align}
	\end{subequations}
	Since $y=u\circ\zeta$, it follows that $y=u\circ\zeta=\D_s\zeta(t,s)|_{s=0}$ and
	\begin{align}\label{DyDelta}
	\D_s K(\zeta(t,s))|_{s=0}=&\lsi[\nb^j]{\zeta}K(\zeta(t,s))\D_s \zeta_j(t,s)|_{s=0}=y_j\lsi[\nb^j]{\zeta}K(\zeta)=D_y\lsi[\Delta]{\zeta},
	\end{align}
	which yields from   $\h_\zeta=I-\lsi[\Delta]{\zeta}^{-1}\lsi[\Delta]{\zeta}$ and the identity 2) in Lemma \ref{lem.1}, that 
	\begin{align}\label{DyH}
	D_y\h_\zeta =&D_y(I-\lsi[\Delta]{\zeta}^{-1}\lsi[\Delta]{\zeta})\no\\
	=&-(D_y\lsi[\Delta]{\zeta}^{-1})\lsi[\Delta]{\zeta}-\lsi[\Delta]{\zeta}^{-1}(D_y\lsi[\Delta]{\zeta})\no\\
	=&\lsi[\Delta]{\zeta}^{-1}(D_y\lsi[\Delta]{\zeta})\lsi[\Delta]{\zeta}^{-1}\lsi[\Delta]{\zeta}-\lsi[\Delta]{\zeta}^{-1}(D_y\lsi[\Delta]{\zeta})\no\\
	=&-\lsi[\Delta]{\zeta}^{-1}(D_y\lsi[\Delta]{\zeta})\h_\zeta \no\\
	=&-\lsi[\Delta]{\zeta}^{-1}\lsi[{[u\cdot\nb,\Delta]}]{\zeta}\h_\zeta.
	\end{align}
	Similarly, we have from the identity 1) in Lemma \ref{lem.1}
	\begin{align}\label{Dyunb}
	D_y\lsi[(u\cdot \nb)]{\zeta}=&y_j\lsi[\nb^j]{\zeta}\lsi[(u\cdot \nb)]{\zeta}=\lsi[\nb^j]{\zeta}\lsi[(u\cdot \nb)]{\zeta}\D_s \zeta_j(t,s)|_{s=0}=\D_s\lsi[(u\cdot \nb)]{\zeta(t,s)}|_{s=0}\no\\
	=&\D_s(u_k(\zeta(t,s))\lsi[\nb^k]{\zeta(t,s)})|_{s=0}\no\\
	=&u_{k,j}(\zeta(t,s))\D_s\zeta^j(t,s)\lsi[\nb^k]{\zeta(t,s)}|_{s=0}+u(\zeta(t,s))\cdot\D_s\lsi[\nb]{\zeta(t,s)}|_{s=0}\no\\
	=&u^j(\zeta)u_{k,j}(\zeta)\lsi[\nb^k]{\zeta}+u(\zeta)\cdot \lsi[{[u\cdot\nb,\nb]}]{\zeta}\no\\
	=&u^j(\zeta)u_{k,j}(\zeta)\lsi[\nb^k]{\zeta}+u_k(\zeta)u_j(\zeta)\lsi[\nb^j]{\zeta}\lsi[\nb^k]{\zeta}\no\\
	&-u^k(\zeta)u_{j,k}(\zeta)\lsi[\nb^j]{\zeta}-u_k(\zeta)u_j(\zeta)\lsi[\nb^k]{\zeta}\lsi[\nb^j]{\zeta}\no\\
	=&0.
	\end{align}
	Hence, it follows from the identity 3) in Lemma \ref{lem.1} that, 
	\begin{align}\label{DyunbnbDelinv}
	&D_y\lsi[{[u\cdot\nb,\nb\Delta^{-1}]}]{\zeta}\no\\
	=&D_y\lsi[{(u\cdot\nb\nb\Delta^{-1})}]{\zeta}-D_y\lsi[{(\nb\Delta^{-1}(u\cdot\nb))}]{\zeta}\no\\
	=&\lsi[{(u\cdot\nb)}]{\zeta}\left((D_y\lsi[\nb]{\zeta})\lsi[\Delta^{-1}]{\zeta}+\lsi[\nb]{\zeta}(D_y\lsi[\Delta^{-1}]{\zeta})\right)\no\\
	&-\left((D_y\lsi[\nb]{\zeta})\lsi[\Delta^{-1}]{\zeta}+\lsi[\nb]{\zeta}(D_y\lsi[\Delta^{-1}]{\zeta})\right)\lsi[{(u\cdot\nb)}]{\zeta}\no\\
	=&\left[\lsi[{(u\cdot\nb)}]{\zeta},\lsi[{[u\cdot\nb,\nb]}]{\zeta}\lsi[\Delta^{-1}]{\zeta}+\lsi[\nb]{\zeta}\lsi[{([u\cdot\nb,\Delta^{-1}]-\h(u\cdot\nb)\Delta^{-1})}]{\zeta}\right]\no\\
	=&\lsi[{[u\cdot\nb,[u\cdot\nb,\nb\Delta^{-1}]]}]{\zeta}-\lsi[{[u\cdot\nb,\nb\h(u\cdot\nb)\Delta^{-1}]}]{\zeta}.
	\end{align}
	
	Therefore, we obtain, from \eqref{DyDelta}-\eqref{DyunbnbDelinv}, that
	\begin{align*}
	\eqref{D2Z.1}=&\Big(\lsi[{[u\cdot\nb,[u\cdot\nb,\nb\Delta^{-1}]]}]{\zeta}-\lsi[{[u\cdot\nb,\nb\h(u\cdot\nb)\Delta^{-1}]}]{\zeta}\no\\
	&-\lsi[{([u\cdot\nb,\nb]\h(u\cdot\nb)\Delta^{-1})}]{\zeta}+\lsi[{(\nb \Delta^{-1}[u\cdot\nb,\Delta]\h(u\cdot\nb)\Delta^{-1})}]{\zeta}\no\\
	&+\lsi[{(\nb \h(u\cdot\nb)\Delta^{-1}[u\cdot\nb,\Delta]\Delta^{-1})}]{\zeta}\Big)\tr(\lsi[D]{\zeta} \dot{\zeta})^2.
	\end{align*}
	The other terms in \eqref{D2Z} including the operator $D_y\big([u\cdot\nb,\nb\Delta^{-1}]-\nb\h(u\cdot\nb)\Delta^{-1}\big)_{\zeta}$, namely \eqref{D2Z.3}, \eqref{D2Z.8} and \eqref{D2Z.12}, have similar identities.
	
	We have, from \eqref{Dstr} and the analogue of \eqref{DyDelta}, that
	\begin{align}
	\eqref{D2Z.1'}=&2\lsi[{\left([u\cdot\nb,\nb\Delta^{-1}]-\nb\h(u\cdot\nb)\Delta^{-1}\right)}]{\zeta}\no\\
	&\qquad  \cdot\tr\left(D_\zeta\dot{\zeta}D_\zeta\dot{y}+D_\zeta\dot{\zeta}([u\cdot\nb,D]_\zeta \dot{\zeta})\right).\label{Dytr}
	\end{align}
	
	By Lemma \ref{lem.1}, one has 
	\begin{align*}
	\eqref{D2Z.2}=&2([u\cdot\nb,\nb\Delta^{-1}]-\nb\h(u\cdot\nb)\Delta^{-1})_{\zeta}\no\\
	&\qquad \cdot\tr\left(D_\zeta\dot{\zeta}D_\zeta\dot{y}+D_\zeta\dot{\zeta}([u\cdot\nb,D]_\zeta \dot{\zeta})\right)=\eqref{Dytr}.
	\end{align*}
	The other terms in \eqref{D2Z} including the operator $D_y\lsi[{\left(\nb\Delta^{-1}\right)}]{\zeta}$, namely \eqref{D2Z.5}, \eqref{D2Z.6}, \eqref{D2Z.7}, \eqref{D2Z.10}, \eqref{D2Z.11}, \eqref{D2Z.14} and \eqref{D2Z.15}, have similar identities.
	
	Since $D_y\zeta_k=y_j\nb_{\zeta}^j\zeta_k=y_j\delta^j_k=y_k$,  $D_y\zeta=y$ and 
	\begin{align*}
	\eqref{D2Z.2'}=&2\lsi[\nb]{\zeta}\lsi[\Delta]{\zeta}^{-1}\tr \Big((D_\zeta\dot{y})^2+[u\cdot\nb,D]_\zeta \dot{\zeta}D_\zeta\dot{y}+D_\zeta\dot{\zeta}[u\cdot\nb,D]_\zeta\dot{y}\no\\
	&+([u\cdot\nb,D]_\zeta\dot{\zeta})([u\cdot\nb,D]_\zeta \dot{\zeta})+D_\zeta\dot{y}([u\cdot\nb,D]_\zeta \dot{\zeta})\no\\
	&+D_\zeta\dot{\zeta}((u\cdot\nb)_{\zeta}[u\cdot\nb,D]_{\zeta} \dot{\zeta}-[u\cdot\nb,D]_{\zeta}(u\cdot\nb)_{\zeta} \dot{\zeta}-D_{\zeta}(u\cdot\nb)_{\zeta} \dot{y})\Big)\no\\
	=&2\lsi[\nb]{\zeta}\lsi[\Delta]{\zeta}^{-1}\tr \Big((D_\zeta\dot{y})^2+2[u\cdot\nb,D]_\zeta \dot{\zeta}D_\zeta\dot{y}+2D_\zeta\dot{\zeta}[u\cdot\nb,D]_\zeta\dot{y}\no\\
	&+([u\cdot\nb,D]_\zeta\dot{\zeta})^2+D_\zeta\dot{\zeta}[u\cdot\nb,[u\cdot\nb,D]]_{\zeta} \dot{\zeta}-D_\zeta\dot{\zeta}(u\cdot\nb)_{\zeta}D_{\zeta} \dot{y}\Big).
	\end{align*}

	Due to
	\begin{align*}
	D_y\zeta^m_{,k}(t,\zeta^{-1}(t,\zeta))=&y_j\nb_\zeta^j\zeta^m_{,k}(t,\zeta^{-1}(t,\zeta)) =y^j\zeta^m_{,kl}(t,\zeta^{-1}(t,\zeta))(\zeta^{-1})^l_{,j}\circ\zeta\\
	=&\zeta^m_{,kl}(t,\zeta^{-1}(t,\zeta))((D\zeta)^{-1})^l_j y^j,
	\end{align*}  
	we have $$D_yD\zeta=D^2\zeta(D\zeta)^{-1}y,$$ 
	and  in general
	\begin{align*}
	D_yD^k\zeta=D^{k+1}\zeta(D\zeta)^{-1}y, \quad \forall k\in\NN\cup\{0\}.
	\end{align*}
	Thus, we get for \eqref{D2Z.4}
	\begin{align*}
	&D_y\tr(\zeta^m_{,kl} \Re_\zeta((\zeta^{-1})^l_{,i}H^k_0))^2
	=D_y(\zeta^m_{,kl} ((D\zeta)^{-1})^l_{i}H^k_0(\zeta)\zeta^i_{,k'l'} ((D\zeta)^{-1})^{l'}_{m}H^{k'}_0(\zeta))\\
	=&2\zeta^m_{,klj}((D\zeta)^{-1})^j_ny^n((D\zeta)^{-1})^l_{i}H^k_0(\zeta)\zeta^i_{,k'l'} ((D\zeta)^{-1})^{l'}_{m}H^{k'}_0(\zeta)\\
	&-2\zeta^m_{,kl}((D\zeta)^{-1})^l_{j}\zeta^j_{,nn'}((D\zeta)^{-1})^{n'}_{i'}y^{i'}((D\zeta)^{-1})^n_{i}H^k_0(\zeta)\zeta^i_{,k'l'} ((D\zeta)^{-1})^{l'}_{m}H^{k'}_0(\zeta)\\
	&+2\zeta^m_{,kl} ((D\zeta)^{-1})^l_{i}H^k_{0,j}(\zeta)y^j\zeta^i_{,k'l'} ((D\zeta)^{-1})^{l'}_{m}H^{k'}_0(\zeta).
	\end{align*}
	Generally, the derivative $D_y$  makes the target to increase one derivative with respect to $\zeta$ for $\zeta$, $y$, $u$ or $H_0$. So we omit the detailed derivation of $D_y$ terms in \eqref{D2Z.5'}, \eqref{D2Z.6'}, \eqref{D2Z.7'}, \eqref{D2Z.9}, \eqref{D2Z.10'}, \eqref{D2Z.11'}, \eqref{D2Z.13}, \eqref{D2Z.14'}, \eqref{D2Z.15'} and \eqref{D2Z.16}.
	
	For any $s$, $\Delta^{-1}: H^s(\Omega)\to H^{s+2}(\Omega)$ and thus $\h:H^s(\Omega)\to H^s(\Omega)$ is bounded. Therefore $D^2Z(\zeta,\dot{\zeta})$ is a bilinear operator in $(y,\dot{y})$ due to $y=u\circ\zeta$ with coefficients depending on at most three derivatives of $(\zeta,\dot{\zeta})$. Moreover, those terms excluding $H_0$ are the first order in $(y,\dot{y})$ times zeroth order in $(y,\dot{y})$, while  the orders of derivatives of $y$ or $H_0$ in each term is at most $2$ for those terms including $H_0$. Thus, by Sobolev's embedding theorem (cf. \cite{Niren59}),
	\begin{align*}
	&\norm{D^2Z(\zeta,\dot{\zeta})((y,\dot{y}),(y,\dot{y}))}_1\\
	\ls &\norm{C(\zeta,\dot{\zeta})(y,\dot{y})(Dy,D\dot{y})}_1+b^2\norm{C(\zeta,\dot{\zeta},H_0)(y,Dy,D^2y)}_1\\
	\ls &\norm{C(\zeta,\dot{\zeta})(y,\dot{y})(Dy,D\dot{y})}_0+\norm{DC(\zeta,\dot{\zeta})(y,\dot{y})(Dy,D\dot{y})}_0\\
	&+\norm{C(\zeta,\dot{\zeta})D(y,\dot{y})(Dy,D\dot{y})}_0+\norm{C(\zeta,\dot{\zeta})(y,\dot{y})D(Dy,D\dot{y})}_0\\
	&+b^2\norm{C(\zeta,\dot{\zeta})(y,Dy,D^2y)}_0+b^2\norm{DC(\zeta,\dot{\zeta})(y,Dy,D^2y)}_0\\
	&+b^2\norm{C(\zeta,\dot{\zeta})D(y,Dy,D^2y)}_0\\
	\ls &C\norm{D^{1/2}(y,\dot{y})}_{0}\norm{D^{3/2}(y,\dot{y})}_0+C\norm{D^{3/2}(y,\dot{y})}_{0}^2\\
	&+C\norm{(y,\dot{y})}_{3/2}\norm{D^2(y,\dot{y})}_{0}+Cb^2\sum_{k=0}^3\norm{y}_k\\
	\ls & C\norm{(y,\dot{y})}_{0}^{1/2}\norm{D(y,\dot{y})}_{0}\norm{D^2(y,\dot{y})}_{0}^{1/2}+C\norm{D(y,\dot{y})}_{0}\norm{D^2(y,\dot{y})}_{0}\\
	&+C\norm{D(y,\dot{y})}_{0}^{1/2}\norm{D^2(y,\dot{y})}_{0}^{3/2}+Cb^2\sum_{k=0}^3\norm{y}_k\\
	\ls &C\norm{(y,\dot{y})}_{1}\norm{D^2(y,\dot{y})}_{0}+C\norm{(y,\dot{y})}_{1}^{1/2}\norm{D^2(y,\dot{y})}_{0}^{3/2}+Cb^2\norm{y}_3,
	\end{align*}
	where $C(\cdots)$ is a smooth function of $(\zeta,\dot{\zeta})$ and their first three derivatives.
	
	The image of $\eta(t)$ is a circle in $H^{s+3}(\Omega;\R^2)$ which is compact and since $H^s(\R^2)$ is an algebra for $s>1$, it then follows
	\begin{align*}
	&\norm{D^2Z(\zeta,\dot{\zeta})((y,\dot{y}),(y,\dot{y}))}_s\\
	\ls &\norm{C(\zeta,\dot{\zeta})(y,\dot{y})(Dy,D\dot{y})}_s+b^2\norm{C(\zeta,\dot{\zeta},H_0)(y,Dy,D^2y)}_s\\
	\ls &C\norm{(y,\dot{y})}_{s}\norm{D^{s+1}(y,\dot{y})}_{0}+C\norm{(y,\dot{y})}_{1}^{(2s-1)/2s}\norm{D^{s+1}(y,\dot{y})}_{0}^{(2s+1)/2s}+Cb^2\norm{y}_{s+2}.
	\end{align*}
	Thus, the desired estimates follow  for all $(\zeta,\dot{\zeta})$ near $(\eta(t),\dot{\eta}(t))$ in $H^{s+3}(\Omega;\R^2)$ with fixed $C$.
\end{proof}

\begin{lemma}\label{lem}
	Let $m>l\gs 1$ be an integer satisfying 	$\norm{D^{m+1}y}_{0}\ls C_m \norm{Dy}_{0}$. Then, it holds
	\begin{align*}
	\norm{y}_{l+1}\ls C\norm{y}_1.
	\end{align*}
\end{lemma}
\begin{proof}
	From Sobolev's embedding theorem (cf. \cite{Niren59}), we have for any $0\ls j<m$,
	\begin{align*}
	\norm{D^{j+1} y}_0\ls C_{j,m}\norm{D^{m+1} y}_0^{j/m}\norm{Dy}_0^{1-j/m}\ls C_{j,m}\norm{Dy}_0.
	\end{align*}
	Therefore, for any $l<m$
	\begin{align*}
	\norm{Dy}_l=&\sum_{j=0}^l\norm{D^{j+1}y}_0\ls \sum_{j=0}^l C_{j,m}\norm{Dy}_0\ls C_m\norm{Dy}_0,
	\end{align*}
	the desired estimate follows then. 
\end{proof}

\begin{corollary}
	Let $m>s$ be an integer satisfying 	$\norm{D^{m}(y,\dot{y})}_{0}\ls C_m \norm{(y,\dot{y})}_{0}$. Then, it holds
	\begin{align*}
	\norm{D^2Z(\zeta,\dot{\zeta})((y,\dot{y}),(y,\dot{y}))}_s	\ls C\norm{(y,\dot{y})}_{0}^{2-1/m}\norm{D^{m}(y,\dot{y})}_{0}^{1/m}+Cb^2\norm{y}_{s+2}.
	\end{align*}
\end{corollary}

\begin{proof}
	From Sobolev's embedding theorem (cf. \cite{Niren59}), it follows that
	\begin{align*}
	&\norm{D^2Z(\zeta,\dot{\zeta})((y,\dot{y}),(y,\dot{y}))}_s\\
	\ls  &C\left(\sum_{k=0}^{\max(s,2)}\norm{D^m(y,\dot{y})}_{0}^{k/m}\norm{(y,\dot{y})}_{0}^{1-k/m}\right)\\
	&\cdot\left(\sum_{l=0}^{s}\norm{D^{m}(y,\dot{y})}_{0}^{l/(m-1)}\norm{D(y,\dot{y})}_{0}^{1-l/(m-1)}\right)+Cb^2\norm{y}_{s+2}\\
	\ls &C\norm{(y,\dot{y})}_{0}\frac{1-\norm{D^m(y,\dot{y})}_{0}^{(\max(s,2)+1)/m}\norm{(y,\dot{y})}_{0}^{-(\max(s,2)+1)/m}}{1-\norm{D^m(y,\dot{y})}_{0}^{1/m}\norm{(y,\dot{y})}_{0}^{-1/m}}\\
	&\cdot \left(\sum_{l=0}^{s}\norm{D^{m}(y,\dot{y})}_{0}^{l/(m-1)}\left(\norm{D^{m}(y,\dot{y})}_{0}^{1/m}\norm{(y,\dot{y})}_{0}^{1-1/m}\right)^{1-l/(m-1)}\right)+Cb^2\norm{y}_{s+2}\\
	\ls &C\norm{(y,\dot{y})}_{0}^{2-1/m}\frac{1-\norm{D^m(y,\dot{y})}_{0}^{(\max(s,2)+1)/m}\norm{(y,\dot{y})}_{0}^{-(\max(s,2)+1)/m}}{1-\norm{D^m(y,\dot{y})}_{0}^{1/m}\norm{(y,\dot{y})}_{0}^{-1/m}}\\
	&\cdot \norm{D^{m}(y,\dot{y})}_{0}^{1/m} \left(\sum_{l=0}^{s}\norm{D^{m}(y,\dot{y})}_{0}^{l/m}\norm{(y,\dot{y})}_{0}^{-l/m}\right)+Cb^2\norm{y}_{s+2}\\
	\ls &C\norm{(y,\dot{y})}_{0}^{2-1/m}\norm{D^{m}(y,\dot{y})}_{0}^{1/m}(1-\norm{D^m(y,\dot{y})}_{0}^{1/m}\norm{(y,\dot{y})}_{0}^{-1/m})^{-2}\\
	&(1-\norm{D^m(y,\dot{y})}_{0}^{(\max(s,2)+1)/m}\norm{(y,\dot{y})}_{0}^{-(\max(s,2)+1)/m})\\
	&(1-\norm{D^m(y,\dot{y})}_{0}^{(s+1)/m}\norm{(y,\dot{y})}_{0}^{-(s+1)/m})+Cb^2\norm{y}_{s+2}\\
	\ls &C\norm{(y,\dot{y})}_{0}^{2-1/m}\norm{D^{m}(y,\dot{y})}_{0}^{1/m}+Cb^2\norm{y}_{s+2},
	\end{align*}
	provided that $\norm{D^{m}(y,\dot{y})}_{0}\ls C_m \norm{(y,\dot{y})}_{0}$.
\end{proof}

\subsection{Decomposition of solutions}

We decompose $y=y_n$ into three parts as that in \cite{Ebin}, each of which will be estimated separately. 

Let $q=:y-\nb \Delta^{-1}\dv y$, the divergence free part of $y$,  and $h$ be a harmonic function satisfying 
\begin{align*}
\langle \nb h,\nu\rangle=\langle q,\nu\rangle, \text{ on } \D\Omega.
\end{align*}
Define
\begin{align}\label{N.def}
N=y-\nb h.
\end{align}
Note that, $h$, as a harmonic function,  is the real part of some holomorphic function $\varphi(z)$.  Therefore, it can be written as  $h=\RE\sum_{j=0}^\infty a_j z^j$. Set $g=\RE\sum_{j=0}^{n-1} a_j z^j$ and $f=h-g=\RE\sum_{j=n}^\infty a_j z^j$. We then decompose $y$ as
\begin{align}\label{y.decomp}
y=\nb f+\nb g+N.
\end{align}
Denote  $\proj_i$, $i=1,2,3$  the projection onto the $i$-th summand of \eqref{y.decomp}.

It is shown in \cite{Ebin} that the summands of \eqref{y.decomp} are orthogonal with respect to the $L^2(\Omega)$ inner product. 

Following that in \cite{Ebin}, the next step is to  decompose equation \eqref{y.eqn}. Denote, for simplicity, 
\begin{align*}
Q=Q(y,\dot{y}):=\int_0^1(1-s)\left(\int_0^s D^2Z(\zeta(\sigma),\dot{\zeta}(\sigma))((y,\dot{y}),(y,\dot{y}))d\sigma\right)ds,
\end{align*}
and $Q_i=\proj_iQ$ for $i=1,2,3$. Whence, 
\begin{align}\label{ydd}
\ddot{y}=DZ(\eta,\dot{\eta})(y,\dot{y})+Q.
\end{align}

From \eqref{weta}, it follows
\begin{align*}
DZ(\eta,\dot{\eta})(y,\dot{y})=&(1-b^2)\big(\lsi[{[u\cdot\nb,\nb\Delta^{-1}]}]{\eta}(-2)-\lsi[\nb]{\eta}\lsi[{(\h(u\cdot\nb)\Delta^{-1})}]{\eta}(-2)\big)\no\\
&+2\lsi[\nb]{\eta}\lsi[\Delta]{\eta}^{-1}\tr\left(Du\circ\eta+\im e^{-\im t}D\dot{y}\right)-\lsi[\nb]{\eta}\lsi[\Delta]{\eta}^{-1}(y^i_{,kl}\Re_\eta(H^k_0H^l_0)_{,i})\no\\
&-e^{\im t}\lsi[\nb]{\eta}\lsi[\Delta]{\eta}^{-1}(2y^j_{,k}\Re_\eta( H^{i}_{0,j} H^k_{0,i}))+b^2 e^{\im t}\eta\no\\
=&-(1-b^2)y+\tilde{\A}y+(\nb\Delta^{-1})_\eta\tr M+b^2e^{\im t}\eta,
\end{align*}
where $\tilde{\A}y=(\nb\h)_\eta(\langle y,\eta\rangle)$ which depends on $\eta$, and 
\begin{align}\label{M}
M=2D_\eta N+2\im e^{-\im t}D\dot{N}-N^i_{,kl}\Re_\eta(H^k_0H^l_0)_{,i'}-2e^{\im t}N^j_{,k}\Re_\eta( H^{i}_{0,j'} H^k_{0,i}).
\end{align}
In fact, by \eqref{N.def}, we have $y=\nb h+N$ with $\Delta h=0$. In view of Lemma \ref{lem.4.19},  there exists a harmonic function $\tilde{h}$ such that $\nb h\circ \eta=\nb \tilde{h}$. Since $\eta^{-1}(t,z)=e^{-\im t}z=e^{-2\im t}\eta$, we have $\D_i h(z)=\D_i\tilde{h}(e^{-2\im t}\eta(z))$. Thus,
$$\D_k\D_i h(z)=\D_i\D_m\tilde{h}(e^{-2\im t}\eta)e^{-2\im t}\eta^m_{,k}=e^{-\im t}\D_i\D_k\tilde{h}(e^{-2\im t}\eta),$$ 
and  
$$\D_l\D_k\D_i h(z)=e^{-2\im t}\D_l\D_k\D_i\tilde{h}(e^{-2\im t}\eta).$$ Therefore,  as in \eqref{wH.1} and \eqref{wH.2}, it yields $$\sum_i\D_ih_{,kl}\Re_\eta(H^k_0H^l_0)_{,i}=0, \text{ and } \sum_j\D_j h_{,k}\Re_\eta( H^{i}_{0,j} H^k_{0,i})=0.$$ 
Also,  $\tr D(\nb h(\eta))=\tr e^{\im t}D\nb h(\eta)=0$ and $\tr(\im e^{-\im t}D\nb \dot{h})=0$ due to $\Delta h=0$.

Therefore,  
\begin{align}\label{y.eq}
\ddot{y}=(b^2-1)y+\tilde{\A}y+(\nb\Delta^{-1})_\eta\tr M+b^2e^{\im t}\eta+Q.
\end{align}

Applying $\proj_3$ to \eqref{y.eq}, noticing that $\proj_3\tilde{A}y=0$, $\proj_3(\nb\h)_\eta=0$ and $ (\proj_1+\proj_2)(\nb\Delta^{-1})_\eta=0$, we obtain, in view of \eqref{M},
\begin{align}\label{eq.N}
\ddot{N}=&(b^2-1)N+(\nb\Delta^{-1})_\eta\tr M+b^2e^{\im t}\eta+Q_3\no\\
=&(b^2-1)N+2(\nb\Delta^{-1})_\eta\dv_\eta N+2\im e^{-\im t}(\nb\Delta^{-1})_\eta\tr(D\dot{N})+b^2e^{\im t}\eta+Q_3\no\\
&-(\nb\Delta^{-1})_\eta(N^i_{,kl}\Re_\eta(H^k_0H^l_0)_{,i})-2e^{\im t}(\nb\Delta^{-1})_\eta (N^j_{,k}\Re_\eta( H^{i}_{0,j} H^k_{0,i}))\no\\
=:&(b^2-1)N+B_1N+B_2\dot{N}+b^2e^{\im t}\eta+Q_3+B_3N+B_4N,
\end{align}
where we have used $\proj_3\eta=\eta$ since 
$$\nb\Delta^{-1}\dv\eta=e^{\im t}\nb\Delta^{-1}\dv(x_1,x_2)=e^{\im t}\nb\Delta^{-1}(2)=e^{\im t}(x_1,x_2)=\eta.$$

Applying $\proj_1+\proj_2$ to \eqref{y.eq} to obtain
\begin{align*}
\ddot{\nb} f+\ddot{\nb} g=&(b^2-1)(\nb f+\nb g)+\tilde{\A}(\nb f+\nb g)+\tilde{\A}N+Q_1+Q_2.
\end{align*}
The terms $(b^2-1)(\nb f+\nb g)$ and $\tilde{\A}(\nb f+\nb g)$ were computed in \eqref{eqf} and \eqref{eqg}, where they were written as $(1-b^2)\A(\nb f+\nb g)+b^2 e^{it}\eta$. Since $\A \bar{z}^k=(k-1)\bar{z}^k$,  $\A$ commutes with $\proj_j$ for $j=1$ or $2$. Denote $\tilde{\A}_j=\proj_j \tilde{\A}$ for $j=1$ or $2$. Thus, in view of the fact $\proj_j\eta=\eta$ for  $j=1$ or $2$ due to $\proj_3\eta=\eta$, we get
\begin{align}
\ddot{\nb} f=&(1-b^2)\A\nb f+\tilde{\A}_1N+Q_1,\label{eq.f}\\
\ddot{\nb} g=&(1-b^2)\A\nb g+\tilde{\A}_2N+Q_2.\label{eq.g}
\end{align}
Thus,  \eqref{ydd} has been  decomposed into the system \eqref{eq.N}, \eqref{eq.f} and \eqref{eq.g}.

\subsection{Estimates of $\nb f$, $\nb g$ and $N$}

Let $y_n=\nb f_n+\nb g_n+N_n$ be the sequence of solutions with the initial data $y_n(0)=0$, $\dot{y}_n(0)=e^{-n^{1/4}}\bar{z}^n$. For any harmonic function $h$ and any real $s$, we introduce the equivalent form of the usual $H^s$-norm, i.e.,
\begin{align}\label{snorm}
\norm{\nb h}_s=\left(\A^s\nb h,\A^s\nb h\right)^{\frac{1}{2}},
\end{align}
where $\A^s$ is defined by $\A^s\bar{z}^k=(k-1)^s\bar{z}^k$. Then, it implies that for any $s$,
\begin{align}\label{eq.532}
\norm{\A\nb g}_s\ls (n-1)\norm{\nb g}_s.
\end{align}

For $\mu,\nu\gs 1$ and $\sigma\gs 2$, denote
\begin{align*}
\begin{aligned}
E_{\mu,b}^\pm=&\norm{\dot{\nb}f\pm \sqrt{1-b^2}\A^{\frac{1}{2}}\nb f}_\mu^2=\norm{\A^\mu(\dot{\nb}f\pm \sqrt{1-b^2}\A^{\frac{1}{2}}\nb f)}_0^2,\\
E_{\mu,b}=&E_{\mu,b}^++E_{\mu,b}^-,\\
F_\sigma=&n\norm{N}_\sigma^2+\norm{\dot{N}}_\sigma^2,\\
G_\nu=&\norm{\dot{\nb}g}_\nu^2+\norm{\A^{\frac{1}{2}}\nb g}_\nu^2.
\end{aligned}
\end{align*}
It is follows from \eqref{snorm} the inequality
\begin{align}\label{rela.E}
E_{\mu,b}^\pm\gs n^{2(\mu-\nu)}E_{\nu,b}^\pm, \quad \text{for } \mu\gs \nu.
\end{align}

\begin{proposition}\label{prop.2}
	Let $\mu\gs 2$. For sufficiently large $n$, the set $E_{\mu,b}^+\gs E_{\mu,b}^-$, $E_{\mu,b}^+\gs \sqrt{n}F_{\mu+1}$, and $E_{\mu+\frac{1}{4},b}^+\gs \frac{2\sqrt{1-b^2}}{(2-b^2)}n^{3/4}G_\mu$ is invariant under the evolution defined by \eqref{eq.N}, \eqref{eq.f} and \eqref{eq.g}. Of course $E_{\mu,b}\ls 2E_{\mu,b}^+$.
\end{proposition}
\begin{proof}
	From \eqref{eq.f},
	\begin{align}\label{dotE}
	\dot{E}_{\mu,b}^\pm=&2(\ddot{\nb}f\pm \sqrt{1-b^2}\A^{\frac{1}{2}}\dot{\nb} f,\dot{\nb}f\pm \sqrt{1-b^2}\A^{\frac{1}{2}}\nb f)_\mu\no\\
	=&2((1-b^2)\A\nb f+\tilde{\A}_1N+Q_1\pm \sqrt{1-b^2}\A^{\frac{1}{2}}\dot{\nb} f,\dot{\nb}f\pm \sqrt{1-b^2}\A^{\frac{1}{2}}\nb f)_\mu\no\\
	=&2(\sqrt{1-b^2}\A^{\frac{1}{2}}(\sqrt{1-b^2}\A^{\frac{1}{2}}\nb f\pm \dot{\nb} f),\dot{\nb}f\pm \sqrt{1-b^2}\A^{\frac{1}{2}}\nb f)_\mu\no\\
	&+2(\tilde{\A}_1N+Q_1,\dot{\nb}f\pm \sqrt{1-b^2}\A^{\frac{1}{2}}\nb f)_\mu\no\\
	=&\pm2\sqrt{1-b^2}E_{\mu+\frac{1}{4},b}^\pm+2(\tilde{\A}_1N+Q_1,\dot{\nb}f\pm \sqrt{1-b^2}\A^{\frac{1}{2}}\nb f)_\mu.
	\end{align}
	Hence,
	\begin{align}\label{Epm}
	\dot{E}_{\mu,b}^+-\dot{E}_{\mu,b}^-=&2\sqrt{1-b^2}\left(E_{\mu+\frac{1}{4},b}^++E_{\mu+\frac{1}{4},b}^-\right)\no\\
	&+2(\tilde{\A}_1N+Q_1,\dot{\nb}f+ \sqrt{1-b^2}\A^{\frac{1}{2}}\nb f)_\mu\no\\
	&-2(\tilde{\A}_1N+Q_1,\dot{\nb}f- \sqrt{1-b^2}\A^{\frac{1}{2}}\nb f)_\mu\no\\
	\gs &2\sqrt{1-b^2}E_{\mu+\frac{1}{4},b}-2\sqrt{E_{\mu,b}}(K\norm{N}_{\mu+1}+\norm{Q_1}_\mu)\no\\
	\gs &2\sqrt{1-b^2}E_{\mu+\frac{1}{4},b}-2\sqrt{E_{\mu,b}}(K\sqrt{F_{\mu+1}/n}+\norm{Q_1}_\mu)\no\\
	\gs &2\sqrt{1-b^2}E_{\mu+\frac{1}{4},b}-2Kn^{-5/4}E_{\mu+\frac{1}{4},b}-2n^{-1/4}\sqrt{E_{\mu+\frac{1}{4},b}}\norm{Q_1}_\mu.
	\end{align}
	By Proposition~\ref{prop.1} and Lemma \ref{lem}, 
	\begin{align*}
	\norm{Q_1}_\mu\ls & C \eps_{n,\mu}^{(2\mu+1)/2\mu}\norm{(y,\dot{y})}_1^{(2\mu-1)/2\mu}+C\eps_{n,\mu}\norm{(y,\dot{y})}_1+Cb^2\norm{y}_{1}\no\\
	\ls & C \eps_{n,\mu}(E_1+F_2+G_1)^{(2\mu-1)/4\mu}+C(\eps_{n,\mu}+b^2)(E_1+F_2+G_1)^{1/2}\no\\
	\ls &C \eps_{n,\mu}(E_{1,b}+n^{-3/4}E_{\frac{5}{4},b}^+)^{(2\mu-1)/4\mu}\no\\
	&+C(\eps_{n,\mu}+b^2)(E_{1,b}+n^{-3/4}E_{\frac{5}{4},b}^+)^{1/2}\no\\
	\ls &C \eps_{n,\mu}(n^{3/2-2\mu}E_{\mu+\frac{1}{4},b}^+)^{(2\mu-1)/4\mu}\no\\
	&+C(\eps_{n,\mu}+b^2)(n^{3/2-2\mu}E_{\mu+\frac{1}{4},b}^+)^{1/2},
	\end{align*}
	where $\eps_{n,\mu}$  can be any sequence of positive constants such that $\lim_{n\to\infty}\eps_n=0$. So,
	\begin{align*}
	\dot{E}_{\mu,b}^+-\dot{E}_{\mu,b}^-
	\gs &2(\sqrt{1-b^2}-Kn^{-5/4}-C(\eps_{n,\mu}+b^2)n^{1/2-\mu})E_{\mu+\frac{1}{4},b}^+\no\\
	&-2C \eps_{n,\mu}n^{-1/4-(2\mu-3/2)(2\mu-1)/4\mu}(E_{\mu+\frac{1}{4},b}^+)^{(4\mu-1)/4\mu},
	\end{align*}
	which is positive for $n$ large. Then, one obtains easily that $E_{\mu,b}\ls 2E_{\mu,b}^+$.
	
	In view of \eqref{eq.N}, we have
	\begin{align*}
	\dot{F}_\sigma=&2n(\dot{N},N)_\sigma+2(\ddot{N},\dot{N})_\sigma\no\\
	=&2(\dot{N},(n-1+b^2)N+B_1N+B_2\dot{N}+b^2e^{\im t}\eta+Q_3+B_3N+B_4N)_\sigma\no\\
	=&2(n-1+b^2)(\dot{N},N)_\sigma+2(\dot{N},B_1N)_\sigma+2(\dot{N},B_2\dot{N})_\sigma+2(\dot{N},B_3N)_\sigma\no\\
	&+2(\dot{N},B_4N)_\sigma+2(\dot{N},Q_3)_\sigma+2b^2e^{2\im t}(\dot{N},(x_1,x_2))_\sigma.
	\end{align*}
	
	It holds that
	\begin{align*}
	|(\dot{N},N)_\sigma|=&|(\A^\sigma\dot{N},\A^\sigma N)|\ls \norm{\A^\sigma\dot{N}}_0\norm{\A^\sigma N}_0
	= \frac{1}{2\sqrt{n}}2\sqrt{n}\norm{\dot{N}}_\sigma\norm{ N}_\sigma\no\\
	\ls& \frac{1}{2\sqrt{n}}(\norm{\dot{N}}_\sigma^2+n\norm{ N}_\sigma^2)=\frac{1}{2\sqrt{n}}F_\sigma,
	\end{align*}
	whereby
	\begin{align*}
	2(n-1+b^2)|(\dot{N},N)_\sigma|\ls \frac{(n-1+b^2)}{\sqrt{n}}F_\sigma\ls \sqrt{n-1+b^2}F_\sigma.
	\end{align*}
	Note that the operator $\nb\Delta^{-1}\dv$ is bounded in the $H^\sigma$-norm, 
	\begin{align*}
	2|(\dot{N},B_1N)_\sigma|=&4|(\A^\sigma \dot{N},\A^\sigma (\nb\Delta^{-1})_\eta\dv_\eta N)|\ls 4C\norm{\dot{N}}_\sigma\norm{ N}_\sigma\ls  Kn^{-1/2}F_\sigma.
	\end{align*}
	Similarly, it holds that
	\begin{align*}
	2|(\dot{N},B_2\dot{N})_\sigma|=2|(\dot{N},2\im e^{-\im t}(\nb\Delta^{-1})_\eta\tr(D\dot{N}))_\sigma|\ls 4C\norm{\dot{N}}_\sigma^2\ls KF_\sigma,
	\end{align*}
	and 
	\begin{align*}
	2|(\dot{N},B_4N)_\sigma|=&4|(\dot{N},(\nb\Delta^{-1})_\eta (N^j_{,k}\Re_\eta( H^{i}_{0,j} H^k_{0,i})))_\sigma|\ls Cb^2n^{-1/2}F_\sigma,
	\end{align*}
	since $H_{0,j}^i=(j-i)b$ for $i,j\in\{1,2\}$ implies $H^{i}_{0,j} H^k_{0,i}=(j-i)(i-k)b^2$. 
	Also, 
	\begin{align*}
	2|(\dot{N},B_3N)_\sigma|=&2|(\dot{N},(\nb\Delta^{-1})_\eta(N^i_{,kl}\Re_\eta(H^k_0H^l_0)_{,i}))_\sigma|\no\\
	=&4b|(i-k)(\dot{N},(\nb\Delta^{-1})_\eta(N^i_{,kl}\Re_\eta H^l_0))_\sigma|\no\\
	=&4b^2|(i-k)(\dot{N},(\nb\Delta^{-1})_\eta((-x_2,x_1)\cdot \nb N^i_{,k}))_\sigma|\no\\
	\ls &Cb^2n^{-1/2}F_{\sigma+\frac{1}{2}}.
	\end{align*}
	Due to $\A (x_1,x_2)=(x_1,x_2)$ and the boundedness of $\Omega$, 
	\begin{align*}
	|(\dot{N},(x_1,x_2))_\sigma|\ls C\norm{\dot{N}}_\sigma\ls KF_\sigma^{1/2}.
	\end{align*}
	Thus, 
	\begin{align}\label{dotF}
	\dot{F}_\sigma\ls (\sqrt{n-1+b^2}+K)F_\sigma +2\sqrt{F_\sigma}(\norm{Q_3}_\sigma+Cb^2)+Cb^2n^{-1/2}F_{\sigma+\frac{1}{2}}.
	\end{align}
	
	Since
	\begin{align*}
	G_\nu=(\A^\nu\dot{\nb}g,\A^\nu\dot{\nb}g)+(\A^{\nu+1/2}\nb g,\A^{\nu+1/2}\nb g),
	\end{align*}
	and \eqref{eq.g}, 
	\begin{align*}
	\dot{G}_\nu=&2(\A^\nu\ddot{\nb}g,\A^\nu\dot{\nb}g)+2(\A^{\nu+1/2}\nb \dot{g},\A^{\nu+1/2}\nb g)\\
	=&2((2-b^2)\A^{\nu+1}\nb g+\A^\nu\tilde{\A}_2N+\A^\nu Q_2,\A^\nu\dot{\nb}g).
	\end{align*}
	
	By \eqref{eq.532}, one has
	\begin{align*}
	(\A^{\nu+1}\nb g,\A^\nu\dot{\nb}g)\ls &  \norm{\A\nb g}_\nu\norm{\dot{\nb}g}_\nu\ls \sqrt{n-1}\norm{\A^{1/2}\nb g}_\nu\norm{\dot{\nb}g}_\nu\no\\
	\ls &\frac{1}{2}\sqrt{n-1}(\norm{\A^{1/2}\nb g}_\nu^2+\norm{\dot{\nb}g}_\nu^2)=\frac{1}{2}\sqrt{n-1}G_\nu.
	\end{align*}
	Noting that $\tilde{\A}_2N=\proj_2\tilde{\A}N=\proj_2(\nb\h)_\eta(N,\eta)$, one has for some constant $K>0$,
	\begin{align*}
	\norm{\tilde{\A}_2N}_\nu\ls K\norm{N}_{\nu+1}.
	\end{align*}
	This implies  
	\begin{align}\label{est.dotG}
	\dot{G}_\nu\ls& (2-b^2)\sqrt{n-1}G_\nu +2K\norm{N}_{\nu+1}\norm{\dot{\nb}g}_\nu +2\norm{Q_2}_\nu\norm{\dot{\nb}g}_\nu\no\\
	\ls & (2-b^2)\sqrt{n-1}G_\nu+2\sqrt{G_\nu}(K\sqrt{F_{\nu+1}/n}+\norm{Q_2}_\nu).
	\end{align}
	
	It follows from  \eqref{dotE}, \eqref{dotF},  \eqref{rela.E} and \eqref{eq.532}, that
	\begin{align}
	&(\dot{E}_{\mu,b}^+-\sqrt{n}\dot{F}_{\mu+1})\no\\
	\gs& 2\sqrt{1-b^2}E_{\mu+\frac{1}{4},b}^++2(\tilde{\A}_1N+Q_1,\dot{\nb}f+ \sqrt{1-b^2}\A^{\frac{1}{2}}\nb f)_\mu\no\\
	&-\sqrt{n}\left((\sqrt{n-1+b^2}+K)F_{\mu+1} +2\sqrt{F_{\mu+1}}(\norm{Q_3}_{\mu+1}+Cb^2)\right)-Cb^2F_{{\mu}+\frac{3}{2}}\label{dotEF.1}\\
	\gs &2\sqrt{1-b^2}E_{\mu+\frac{1}{4},b}^+-2\sqrt{E_{\mu,b}^+}(\norm{\tilde{\A}_1N}_\mu+\norm{Q_1}_\mu)+\eqref{dotEF.1}\no\\
	\gs &2\sqrt{1-b^2}E_{\mu+\frac{1}{4},b}^+-2K\sqrt{E_{\mu,b}^+}\sqrt{F_{\mu+1}/n}\no\\
	&-2C\sqrt{E_{\mu,b}^+} \eps_{n,\mu}^{(2\mu+1)/2\mu}(E_{1,b}+n^{5/4-2\mu}E_{\mu+\frac{1}{4},b}^+)^{(2\mu-1)/4\mu}\no\\
	&-2C\sqrt{E_{\mu,b}^+}(\eps_{n,\mu}+b^2)(E_{1,b}+n^{5/4-2\mu}E_{\mu+\frac{1}{4},b}^+)^{1/2}\no\\
	&-\sqrt{n}(\sqrt{n-1+b^2}+K)F_{\mu+1}-2\sqrt{n}\sqrt{F_{\mu+1}}\norm{Q_3}_{\mu+1}\no\\
	&-Cb^2(2\sqrt{nF_{\mu+1}}+F_{{\mu}+\frac{3}{2}})\no\\
	\gs &2\sqrt{1-b^2}E_{\mu+\frac{1}{4},b}^+-2Kn^{-3/4}E_{\mu,b}^+\no\\
	&-2C\sqrt{E_{\mu,b}^+} \eps_{n,\mu}^{(2\mu+1)/2\mu}(n^{2(1-\mu)}E_{\mu,b}+n^{5/4-2\mu}E_{\mu+\frac{1}{4},b}^+)^{(2\mu-1)/4\mu}\no\\
	&-2C\sqrt{E_{\mu,b}^+}(\eps_{n,\mu}+b^2)(n^{2(1-\mu)}E_{\mu,b}+n^{5/4-2\mu}E_{\mu+\frac{1}{4},b}^+)^{1/2}\no\\
	&-(\sqrt{n-1+b^2}+K)E_{\mu,b}^+-2n^{1/4}\sqrt{E_{\mu,b}^+}\norm{Q_3}_{\mu+1}\no\\
	&-Cb^2(2n^{1/4}\sqrt{E_{\mu,b}^+}+n^{-1/2}E_{{\mu}+\frac{1}{2},b}^+)\no\\
	\gs &2\sqrt{1-b^2}E_{\mu+\frac{1}{4},b}^+-(\sqrt{n-1+b^2}+K+2Kn^{-3/4})n^{-1/2}E_{\mu+\frac{1}{4},b}^+\no\\
	&-C \eps_{n,\mu}^{(2\mu+1)/2\mu}n^{2(1-\mu)(2\mu-1)/4\mu}(n^{-1/2}E_{\mu+\frac{1}{4},b}^+)^{1-1/4\mu}\no\\
	&-C \eps_{n,\mu}^{(2\mu+1)/2\mu}n^{7/8-\mu-5/16\mu}(E_{\mu+\frac{1}{4},b}^+)^{(2\mu-1)/4\mu}\no\\
	&-C(\eps_{n,\mu}+b^2)n^{(3-4\mu)/4}E_{\mu+\frac{1}{4},b}^+-2n^{1/4}\sqrt{E_{\mu,b}^+}\norm{Q_3}_{\mu+1}\no\\
	&-Cb^2(2\sqrt{E_{\mu+\frac{1}{4},b}^+}+n^{-1/4}((n-1)/n)^{1/4}E_{{\mu}+\frac{1}{4},b}^+). \label{EF}
	\end{align}
	
	Due to Proposition~\ref{prop.1} and Lemma \ref{lem}, 
	\begin{align}\label{Q3}
	\norm{Q_3}_{\mu+1}\ls & C \eps_{n,\mu}^{(2\mu+3)/2(\mu+1)}\norm{(y,\dot{y})}_1^{(2\mu+1)/2(\mu+1)}+C\eps_{n,\mu}\norm{(y,\dot{y})}_{1}+Cb^2\norm{y}_{1}\no\\
	\ls & C \eps_{n,\mu}^{(2\mu+3)/2(\mu+1)}(E_{1,b}+F_2+G_1)^{(2\mu+1)/4(\mu+1)}\no\\
	&+C(\eps_{n,\mu}+b^2)(E_{1,b}+F_2+G_1)^{1/2}\no\\
	\ls &C \eps_{n,\mu}^{(2\mu+3)/2(\mu+1)}(E_{1,b}+n^{5/4-2\mu}E_{\mu+\frac{1}{4},b}^+)^{(2\mu+1)/4(\mu+1)}\no\\
	&+C(\eps_{n,\mu}+b^2)(E_{1,b}+n^{5/4-2\mu}E_{\mu+\frac{1}{4},b}^+)^{1/2}.
	\end{align}
	Thus,
	\begin{align*}
	&2n^{1/4}\sqrt{E_{\mu,b}^+}\norm{Q_3}_{\mu+1}\no\\
	\ls &C\eps_{n,\mu}^{(2\mu+3)/2(\mu+1)}n^{1/4}\sqrt{E_{\mu,b}^+}(E_{1,b}+n^{5/4-2\mu}E_{\mu+\frac{1}{4},b}^+)^{(2\mu+1)/4(\mu+1)}\no\\
	&+C(\eps_{n,\mu}+b^2)n^{1/4}\sqrt{E_{\mu,b}^+}(E_{1,b}+n^{5/4-2\mu}E_{\mu+\frac{1}{4},b}^+)^{1/2}\no\\
	\ls &C\eps_{n,\mu}^{(2\mu+3)/2(\mu+1)}\sqrt{E_{\mu+\frac{1}{4},b}^+}(n^{3/2-2\mu}E_{\mu+\frac{1}{4},b}+n^{5/4-2\mu}E_{\mu+\frac{1}{4},b}^+)^{(2\mu+1)/4(\mu+1)}\no\\
	&+C(\eps_{n,\mu}+b^2)\sqrt{E_{\mu+\frac{1}{4},b}^+}(n^{3/2-2\mu}E_{\mu+\frac{1}{4},b}+n^{5/4-2\mu}E_{\mu+\frac{1}{4},b}^+)^{1/2}\no\\
	\ls &C\eps_{n,\mu}^{(2\mu+3)/2(\mu+1)}n^{(3-4\mu)(2\mu+1)/8(\mu+1)}(E_{\mu+\frac{1}{4},b}^+)^{1-1/4(\mu+1)}\no\\
	&+C(\eps_{n,\mu}+b^2)n^{(3-4\mu)/4}E_{\mu+\frac{1}{4},b}^+.
	\end{align*}
	Hence,  for large $n$, 
	$$\dot{E}_{\mu,b}^+-\sqrt{n}\dot{F}_{\mu+1}>0. $$
	
	From \eqref{dotE}, \eqref{est.dotG}, we get
	\begin{align}\label{EG}
	&\dot{E}_{\mu+\frac{1}{4},b}^+-\frac{2\sqrt{1-b^2}}{(2-b^2)}n^{3/4}\dot{G}_\mu\no\\
	\gs& 2\sqrt{1-b^2}E_{\mu+\frac{1}{2},b}^++2(\tilde{\A}_1N+Q_1,\dot{\nb}f+ \sqrt{1-b^2}\A^{\frac{1}{2}}\nb f)_{\mu+\frac{1}{4}}\no\\
	&-2\sqrt{1-b^2}n^{3/4}\sqrt{n-1}G_\mu-\frac{4\sqrt{1-b^2}}{(2-b^2)}n^{3/4}\sqrt{G_\mu}(K\sqrt{F_{\mu+1}/n}+\norm{Q_2}_\mu)\no\\
	\gs& 2\sqrt{1-b^2}E_{\mu+\frac{1}{2},b}^+ -2\sqrt{E_{\mu+\frac{1}{4},b}^+}(\sqrt{F_{\mu+\frac{5}{4}}/n}+\norm{Q_1}_{\mu+\frac{1}{4}})\no\\
	&-2\sqrt{1-b^2}\sqrt{n-1}E_{\mu+\frac{1}{4},b}^+-2n^{3/8}\sqrt{E_{\mu+\frac{1}{4},b}^+}(K\sqrt{F_{\mu+1}/n}+\norm{Q_2}_\mu)\no\\
	\gs &2\sqrt{1-b^2}E_{\mu+\frac{1}{2},b}^+ -2(n^{-3/4}E_{\mu+\frac{1}{4},b}^++\sqrt{E_{\mu+\frac{1}{4},b}^+}\norm{Q_1}_{\mu+\frac{1}{4}})\no\\
	&-2\sqrt{1-b^2}\sqrt{n-1}E_{\mu+\frac{1}{4},b}^+-2n^{3/8}\sqrt{E_{\mu+\frac{1}{4},b}^+}(Kn^{-3/4}\sqrt{E_{\mu,b}^+}+\norm{Q_2}_\mu)\no\\
	\gs &2\sqrt{1-b^2}E_{\mu+\frac{1}{2},b}^+ -2(n^{-5/4}E_{\mu+\frac{1}{2},b}^++n^{-1/4}\sqrt{E_{\mu+\frac{1}{2},b}^+}\norm{Q_1}_{\mu+\frac{1}{4}})\no\\
	&-2\sqrt{1-b^2}\sqrt{(n-1)/n}E_{\mu+\frac{1}{2},b}^+-2n^{1/8}\sqrt{E_{\mu+\frac{1}{2},b}^+}(Kn^{-5/4}\sqrt{E_{\mu+\frac{1}{2},b}^+}+\norm{Q_2}_\mu)\no\\
	\gs & (2\sqrt{1-b^2}(1-\sqrt{(n-1)/n})-2n^{-5/4}-2Kn^{-9/8})E_{\mu+\frac{1}{2},b}^+\no\\
	&-2n^{-1/4}\sqrt{E_{\mu+\frac{1}{2},b}^+}\norm{Q_1}_{\mu+\frac{1}{4}}-2n^{1/8}\sqrt{E_{\mu+\frac{1}{2},b}^+}\norm{Q_2}_\mu.
	\end{align}
	
	Using Proposition~\ref{prop.1} and Lemma \ref{lem} to obtain
	\begin{align}\label{Q2}
	\norm{Q_2}_\mu\ls &C \eps_{n,\mu}^{(2\mu+1)/2\mu}(E_{1,b}+n^{-3/4}E_{\frac{5}{4},b}^+)^{(2\mu-1)/4\mu}\no\\
	&+C(\eps_{n,\mu}+b^2)(E_{1,b}+n^{-3/4}E_{\frac{5}{4},b}^+)^{1/2}\no\\
	\ls &C \eps_{n,\mu}^{(2\mu+1)/2\mu}(n^{1-2\mu}E_{\mu+\frac{1}{2},b}^++n^{3/4-2\mu}E_{\mu+\frac{1}{2},b}^+)^{(2\mu-1)/4\mu}\no\\
	&+C(\eps_{n,\mu}+b^2)(n^{1-2\mu}E_{\mu+\frac{1}{2},b}^++n^{3/4-2\mu}E_{\mu+\frac{1}{2},b}^+)^{1/2}\no\\
	\ls &C \eps_{n,\mu}n^{-(2\mu-1)^2/4\mu}(E_{\mu+\frac{1}{2},b}^+)^{(2\mu-1)/4\mu}\no\\
	&+C(\eps_{n,\mu}+b^2)n^{1/2-\mu}(E_{\mu+\frac{1}{2},b}^+)^{1/2}.
	\end{align}
	Similarly,
	\begin{align}\label{Q11}
	\norm{Q_1}_{\mu+\frac{1}{4}}\ls &C \eps_{n,\mu}(E_{1,b}+n^{-3/4}E_{\frac{5}{4},b}^+)^{(4\mu-1)/2(4\mu+1)}\no\\
	&+C(\eps_{n,\mu}+b^2)(E_{1,b}+n^{-3/4}E_{\frac{5}{4},b}^+)^{1/2}\no\\
	\ls &C \eps_{n,\mu} n^{-(2\mu-1)(4\mu-1)/2(4\mu+1)}(E_{\mu+\frac{1}{2},b}^+)^{(4\mu-1)/2(4\mu+1)}\no\\
	&+C(\eps_{n,\mu}+b^2)n^{1/2-\mu}(E_{\mu+\frac{1}{2},b}^+)^{1/2}.
	\end{align}
	Thus, combining \eqref{EG}, \eqref{Q2} and \eqref{Q11}, 
	\begin{align*}
	&\dot{E}_{\mu+\frac{1}{4},b}^+-\frac{2\sqrt{1-b^2}}{(2-b^2)}n^{3/4}\dot{G}_\mu\no\\
	\gs& (2\sqrt{1-b^2}(1-\sqrt{(n-1)/n})-2n^{-5/4}-2Kn^{-9/8})E_{\mu+\frac{1}{2},b}^+\no\\
	&-C \eps_{n,\mu} n^{-1/4-(2\mu-1)(4\mu-1)/2(4\mu+1)}(E_{\mu+\frac{1}{2},b}^+)^{4\mu/(4\mu+1)}\no\\
	&-C(\eps_{n,\mu}+b^2)n^{1/4-\mu}E_{\mu+\frac{1}{2},b}^+-C(\eps_{n,\mu}+b^2)n^{5/8-\mu}E_{\mu+\frac{1}{2},b}^+\no\\
	&-C \eps_{n,\mu}n^{1/8-(2\mu-1)^2/4\mu}(E_{\mu+\frac{1}{2},b}^+)^{(4\mu-1)/4\mu}\no\\
	\gs &\frac{1}{n}\Big[(\sqrt{1-b^2}-n^{-1/4}-Kn^{-1/8}-C(\eps_{n,\mu}+b^2)n^{13/8-\mu})E_{\mu+\frac{1}{2},b}^+\no\\
	&\quad -C\eps_{n,\mu} n^{3/4-(2\mu-1)(4\mu-1)/2(4\mu+1)}(E_{\mu+\frac{1}{2},b}^+)^{4\mu/(4\mu+1)}\no\\
	&\quad-C \eps_{n,\mu}n^{9/8-(2\mu-1)^2/4\mu}(E_{\mu+\frac{1}{2},b}^+)^{(4\mu-1)/4\mu}\Big],
	\end{align*}
	which is positive for large $n$  provided $\mu\gs 2$.
	
	Therefore, we have proved the desired results.
\end{proof}

\begin{theorem}\label{thm.evo}
	Let $\mu\gs 2$ and $|b|\ll 1$. For large $n$, $E_{\mu,b}^+(t)\gs E_{\mu,b}^+(0)e^{\sqrt{1-b^2}\sqrt{n}t}$ for $E$, $F$ and $G$ in the invariant set of Proposition \ref{prop.2}.
\end{theorem}
\begin{proof}
	It follows from \eqref{dotE} and similar arguments for \eqref{Epm}, and \eqref{rela.E} that
	\begin{align*}
	\dot{E}_{\mu,b}^+\gs \sqrt{1-b^2}E_{\mu+\frac{1}{4},b}^+\gs \sqrt{1-b^2}n^{1/2}E_{\mu,b}^+,
	\end{align*}
	which yields the desired estimates.
\end{proof}

\subsection{Ill-posedness}
The ill-posedness theorem is the following:
\begin{theorem}\label{ill-posedness} Suppose $|b|\ll 1$. For the initial data $\Omega_0=\{z\in {\C}, |z|\ls 1\}$,  $\zeta_n(0)= z$ , $\dot{z}_n(0)=e^{-n^{1/4}}\bar{z}^n-\im z$ ($n\gs 2$) and $H_0(z)=\im bz$ for $z\in \Omega_0$,  $\zeta_n(t)$ ($n\gs 2$) be the solution to  problem \eqref{eqn1} in some time interval $[0, T]$. 
	Then, for $\mu\gs 2$,    $\norm{(y_n(0),\dot{y}_n(0))}_\mu\to 0$  as  $\ n\to \infty,$ 
	but for any $t>0$,
	$$\norm{(y_n(t),\dot{y}_n(t))}_\mu\to\infty, \ {\rm as}\ n\to \infty, $$
	where 
	$$y_n(t)=\zeta_n(t)-\eta(t),$$
	and $\eta(t,z)=e^{\im t}z$ is a special solution to  \eqref{eqn1}. 
\end{theorem}
\begin{proof}
	Decompose $y_n$ into $\nb f+\nb g+N$, then we define $E_{\mu,b}^\pm$, $F_\mu$ and $G_\mu$ which are in the invariant set of Proposition \ref{prop.2} at time $t=0$. Hence, $E_{\mu,b}^+(t)\gs E_{\mu,b}^+(0)e^{\sqrt{1-b^2}\sqrt{n}t}$,  by Theorem \ref{thm.evo}.  However,
	\begin{align*}
	E_{\mu,b}^+(0)=&\int_\Omega |\A^\mu(\dot{\nb}f\pm \sqrt{1-b^2}\A^{\frac{1}{2}}\nb f)|^2_{t=0}dx\\
	=&\int_\Omega |(n-1)^\mu e^{-n^{1/4}}\bar{z}^n|^2dx\\
	=&\int_0^{2\pi}d\theta \int_0^1 e^{-2n^{1/4}}(n-1)^{2\mu}r^{2n}dr\\
	=&\frac{2\pi (n-1)^{2\mu}}{2n+1}e^{-2n^{1/4}}.
	\end{align*}
	Therefore, 
	$$E_{\mu,b}^+(t)\gs \frac{2\pi (n-1)^{2\mu}}{2n+1}e^{\sqrt{1-b^2}\sqrt{n}t-2n^{1/4}},$$
	which tends to $\infty$ for any $t>0$ as $n\to \infty$. Therefore,  $$\norm{(y_n(t),\dot{y}_n(t))}_\mu\to\infty$$ for $t>0$. 
\end{proof}

The above theorem shows the  ill-posedness for that solutions $\zeta_n(t)$ would exist and converge to $\eta(t)$ on some interval $[0,T]$
if \eqref{eqn1} were well posed. 

The final remark is that the analysis of this paper is 
uniform in the vacuum permeability $\mu_0>0$ as long as we take $b^2\ll \min(1,\mu_0)$ in the proofs.

\medskip

\noindent \textit{Acknowledgements.} Hao was partially supported by NSFC grants 11671384 and 11628103. Luo was  supported by a GRF grant CityU 11303616 of RGC (Hong Kong).

\end{document}